\pdfoutput=1
\documentclass [12pt]{amsart}
\usepackage[utf8]{inputenc}

\usepackage{amsmath,amssymb,amsthm,amsfonts}
\usepackage{mathtools}

\usepackage[english]{babel}
\usepackage{comment}
\usepackage{hyperref}
\usepackage{bbm}
\usepackage{mathrsfs}
\usepackage{textcomp}
\usepackage{tikz,graphicx,color}
\usepackage{tikz-cd}
\usepackage{tikz-3dplot}
\usetikzlibrary{calc}
\usetikzlibrary{arrows}
\usetikzlibrary{shapes}
\usetikzlibrary{patterns}
\usetikzlibrary{positioning}
\usetikzlibrary{arrows.meta}
\usetikzlibrary{decorations.markings}
\usepackage{epstopdf}

\usepackage{enumerate}
\usepackage[arrow]{xy}
\usepackage{diagbox}
\usepackage{subfig}
\usepackage{chngcntr} 
\usepackage{arcs}
\usepackage{xcolor}
\usepackage{tabu}
\usepackage{booktabs}

\usepackage[margin=1.2in]{geometry}
\usepackage{xspace}

\usepackage[outline]{contour}
\contourlength{1.5pt}

\newtheorem{theorem}{Theorem}[section]

\newtheorem{lemma}[theorem]{Lemma}
\newtheorem{proposition}[theorem]{Proposition}
\newtheorem{corollary}[theorem]{Corollary}
\theoremstyle{definition}
\newtheorem{remark}[theorem]{Remark}
\theoremstyle{definition}
\newtheorem{definition}[theorem]{Definition}

\theoremstyle{definition}

\theoremstyle{definition}

\def\Fcal{\mathcal{F}}\def\Lcal{\mathcal{L}}\def\Mcal{\mathcal{M}}\def\Pcal{\mathcal{P}}\def\Rcal{\mathcal{R}}\def\Scal{\mathcal{S}}\def\Tcal{\mathcal{T}}

% ...

\def\C{\mathbb{C}}
\def\R{\mathbb{R}}

\def\Z{\mathbb{Z}}

% ...

\newcommand\parr[1]{{({#1})}}
% ...

\def\<{{\langle}}
\def\>{{\rangle}}

\def\toi{{\xhookrightarrow{i}}}

\def\det{{ \operatorname{det}}}

\def\Conv{ \operatorname{Conv}}

\def\triang{\tau}
\def\T{T_{\svt}} 

\def\SVTGE{E(\SVTG)} 
\def\SVT{\Gamma} 
\def\SVTE{E(\SVT)} 
\def\svt{\mathbb{T}}
\def\ECTG{\Lambda_G} 
\def\ECT{\Lambda} 
\def\ECTE{E(\Lambda)} 
\def\ect{\mathcal{E}}

\newcommand{\too}[1]{\xrightarrow[#1]{}}
\def\dotUp{\circ}
\def\dotDown{\bullet}
\def\toi{\too{i}}
\def\toidot{\toi\dotDown}
\def\dottoi{\dotUp\toi}
\def\toiprime{\too{i'}}
\def\toiprimedot{\toiprime\dotDown}

\def\toj{\too{j}}
\def\tojdot{\toj\dotDown}

\def\tok{\too{k}}
\def\tokdot{\tok\dotDown}
\def\dottok{\dotUp\tok}

\def\DZ{\Delta^\Z}

\def\axCardinalityK{\hyperref[ax:cardinality_K]{(T1)}\xspace}
\def\axExistenceK{\hyperref[ax:existence_K]{(T2)}\xspace}
\def\axContainmentK{\hyperref[ax:containment_K]{(T3)}\xspace}
\def\axHexagonK{\hyperref[ax:hexagon_K]{(T4)}\xspace}

\def\Forests{\Fcal}
\def\Matchings{\Pcal\Mcal}
\def\Trees{\Tcal}
\def\RM{\Rcal\Scal\Mcal}
\def\LM{\Lcal\Scal\Mcal}

\def\LD{\operatorname{LD}}
\def\RD{\operatorname{RD}}

\def\LDp{\LD}
\def\RDp{\RD}
\def\LDm{\LD^-}
\def\RDm{\RD^-}

\def\MP{M} 

\def\axCardinality{\hyperref[ax:cardinality]{(T1')}\xspace}
\def\axExistence{\hyperref[ax:existence]{(T2')}\xspace}
\def\axContainment{\hyperref[ax:containment]{(T3')}\xspace}
\def\axHexagon{\hyperref[ax:hexagon]{(T4')}\xspace}

\newcommand\hidefigure[1]{#1} 

\newcommand\donthidefigure[1]{#1}

\newcommand{\Comment}[1]{{\color{red} \sf $\diamondsuit$ [#1]}}
\newcommand{\Gleb}[1]{{\color{blue} \sf [#1]}}
\newcommand{\Pasha}[1]{{\color{green!70!black} \sf [#1]}}

\begin{document}
\numberwithin{equation}{section}

\title{Trianguloids and Triangulations of Root Polytopes}
\author{Pavel Galashin}
\address{Department of Mathematics, Massachusetts Institute of Technology,
  Cambridge, MA 02139, USA}
\email{{\href{mailto:galashin@mit.edu}{galashin@mit.edu}}}
\email{{\href{mailto:apost@math.mit.edu}{apost@math.mit.edu}}}
\author{Gleb Nenashev}
\address{Department of Mathematics, Stockholm University, SE-106 91, Stockholm, Sweden}
\email{{\href{mailto:nenashev@math.su.se}{nenashev@math.su.se}}}
\author{Alexander Postnikov}
\date{\today}



\begin{abstract}
Triangulations of a product of two simplices and, more generally,
of root polytopes are closely related to Gelfand-Kapranov-Zelevinsky's theory of
discriminants, to tropical geometry, tropical oriented matroids, and to generalized
permutohedra.
We introduce a new approach to these objects, identifying a triangulation of a root polytope with a certain bijection between lattice points
of two generalized permutohedra. 
In order to study such bijections, we define {\it trianguloids} as
edge-colored graphs satisfying simple local axioms.
We prove that trianguloids are in bijection with triangulations
of root polytopes.
\end{abstract}

\maketitle


\section{Introduction}
\def\DD{\Delta^{m-1}\times\Delta^{n-1}}

Triangulations of a product $\DD$ of two simplices have been studied for the last several decades, see e.g.~\cite[Section~8]{ES}, \cite[Section~16.3]{FF}, or~\cite{BCS}. Since then, these objects have naturally appeared in many diverse contexts in combinatorics and algebraic geometry~\cite{SZ,BZ,GKZ,BB,SantosDisc}. They have recently become a subject of active research due to their close relationship to tropical geometry~\cite{DS,AD} and Schubert calculus~\cite{AB}.

Triangulations of $\DD$ are in bijection with various objects, such as \emph{fine mixed subdivisions} of $n\Delta^{m-1}$~\cite{SantosCayley,HRS}, \emph{tropical oriented matroids}~\cite{AD,OYtom}, \emph{tropical pseudohyperplane arrangements}~\cite{Horn}, \emph{matching ensembles}~\cite{BZ,OYmatching}, and  \emph{compatible families of trees}~\cite{Postnikov}. In particular, it was shown in~\cite{Postnikov} that a triangulation of $\DD$ gives rise to a bijection between lattice points of $(n-1)\Delta^{m-1}$ and of $(m-1)\Delta^{n-1}$. More generally, for an arbitrary connected subgraph $G$ of the complete bipartite graph $K_{m,n}$, \cite{Postnikov} introduced the \emph{root polytope} $Q_G$ which specializes to $\DD$ for $G=K_{m,n}$. He showed that a triangulation $\triang$ of $Q_G$ corresponds to a fine mixed subdivision of a \emph{generalized permutohedron} $P_G$ and  yields a bijection $\phi_\triang$ between the lattice points of two \emph{trimmed generalized permutohedra} $P_G^-$ and $P_{G^\ast}^-$. One of the main motivations for this project was to study the bijections $\phi_\triang$ that arise in this way. For example, it follows as a simple consequence of our approach that a triangulation $\triang$ of $Q_G$ can be uniquely reconstructed from the corresponding bijection $\phi_\triang$. See~\cite[Theorem~5]{BZ} and~\cite[Conjecture~6.11]{SZ} for related results.

Another motivation comes from the work of Ardila and Billey~\cite{AB} who described the matroid formed by the lines in the intersection lattice of $m$ generic complete flags in $\C^n$. They raised the \emph{Spread Out Simplices Conjecture} which characterizes the positions of special simplices in a mixed subdivision of $n\Delta^{m-1}$.

We introduce certain edge-colored directed graphs called \emph{trianguloids} (an example shown in Figure~\ref{fig:trianguloid}). We define them axiomatically and show that they are in a natural bijective correspondence with triangulations of $\DD$ (more generally, of $Q_G$). One aspect in which trianguloids differ from  some of the objects listed above is that our axioms are \emph{local}, and in addition, we make no assumptions on the compatibility of the trees appearing in a triangulation. We hope that these properties of our axioms may produce a way of resolving the Spread Out Simplices Conjecture. 



\subsection*{Outline}
We introduce root polytopes and their triangulations in Section~\ref{sec:preliminaries}, and then we state our main results for the case $Q_G=\DD$ in Section~\ref{sec:main_K}. We explain the relationship between trianguloids and various objects that have been studied before in Section~\ref{sec:history}. We then formulate our main result for the case of arbitrary $G$ (Theorem~\ref{thm:svt_implies_triang}) in Section~\ref{sec:main}. 

For the remaining part of the paper, we concentrate on the proofs. In Section~\ref{sec:triang_implies_svt}, we show that each triangulation gives rise to a trianguloid. In Section~\ref{sec:svt_implies_triang}, we show that each trianguloid gives rise to a triangulation. 
 Finally, in Section~\ref{sec:bijections_different} we use our machinery to give simple proofs to Theorems~\ref{thm:bijections_different_K} and~\ref{thm:bijections_different} that a triangulation $\triang$ can be reconstructed from $\phi_\triang$. 

\subsection*{Acknowledgment}
This material is based upon work supported by the National Science Foundation under Grant \textnumero DMS-1440140 while authors were in residence at the Mathematical Sciences Research Institute in Berkeley, California, during the program ``Geometric and Topological Combinatorics'' in Fall~2017.

\section{Preliminaries}\label{sec:preliminaries}
Let us fix integers $m,n\geq 1$ and consider the sets $[m]:=\{1,2,\dots,m\}$ and $[\bar n]:=\{\bar1,\bar2,\dots,\bar n\}$. Define the \emph{complete bipartite graph} $K_{m,n}$ to be a simple graph with vertex set $V:=[m]\cup[\bar n]$ and edge set $\{(i,\bar j)\mid i\in [m], \bar j\in[\bar n]\}$. We identify subgraphs of $K_{m,n}$ with their sets of edges. Clearly, a graph $G\subset K_{m,n}$ is determined by the sets $N_{\bar1}(G),N_{\bar2}(G),\dots,N_{\bar n}(G)\subset[m]$, where $N_{\bar j}(G):=\{i\in [m]\mid (i,\bar j)\in G\}$ is the \emph{neighborhood} of $\bar j\in[\bar n]$ in $G$. Throughout, we fix a connected $G\subset K_{m,n}$ and pay special attention to the case $G=K_{m,n}$.

Consider an $(m+n)$-dimensional real vector space with basis
\[e_1,e_2,\dots,e_m,e_{\bar1},e_{\bar2},\dots,e_{\bar n}\in\R^{m+n}.\]
For a set $I\subset[m]$, we define $\Delta_I\subset \R^m$ to be the convex hull of the points $\{e_i\mid i\in I\}$. Thus $\Delta_I$ is an $(|I|-1)$-dimensional simplex. We denote by $\Delta^{m-1}:=\Delta_{[m]}$ the standard $(m-1)$-dimensional simplex.

\def\DD{\Delta_{[m]}\times\Delta_{[\bar n]}}

The \emph{root polytope} $Q_G\subset \R^{m+n}$ was introduced in~\cite{Postnikov} as the convex hull of the points $e_i-e_{\bar j}$ for all $(i,\bar j)\in G$. When $G=K_{m,n}$ is complete, $Q_G$ is the direct product of two simplices $\DD$, see~\cite[Section~12]{Postnikov}.

We now recall the notions of \emph{Minkowski sum} and \emph{Minkowski difference}.

\def\PGpm{P_G^\pm}
\begin{definition}
  For two subsets $A,B\subset \R^k$, define
  \[ A+B:=\{a+b\mid a\in A,\ b\in B\},\quad A-B:=\{c\in\R^k\mid c+B\subset A\}.\]
\end{definition}
Note that $A-B$ can be empty, so it is not always the case that $(A-B)+B=A$. However, if $A$ and $B$ are convex polytopes then it is true that $(A+B)-B=A$, see~\cite[Lemma~11.1]{Postnikov}. We define three polytopes $P_G,P_G^-,\PGpm\in \R_{\geq0}^m$ associated with $G$ as
\[P_G:=\sum_{\bar j\in[\bar n]}\Delta_{N_{\bar j}(G)},\quad P_G^-:=P_G-\Delta_{[m]},\quad \PGpm:=\left(P_G+(-\Delta_{[m]})\right)\cap \R_{\geq0}^m.\]
Here $-\Delta_{[m]}$ is the convex hull of $\{-e_i\mid i\in [m]\}$.

Thus $P_G$ is a \emph{generalized permutohedron} and $P_G^-$ is a \emph{trimmed generalized permutohedron} in the sense of~\cite{Postnikov}. The polytope $\PGpm$ contains $P_G^-$.  In the case $G=K_{m,n}$, we have $N_{\bar j}(G)=[m]$ for each $\bar j\in[\bar n]$, so $P_G=n\Delta_{[m]}$ and $P_G^-=\PGpm=(n-1)\Delta_{[m]}$ are just dilated $(m-1)$-simplices. 

Our main focus is the set of \emph{triangulations} of the root polytope $Q_G$. For a subgraph $F\subset G$, we let $\Delta_F$ be the convex hull of $e_i-e_{\bar j}$ for all $(i,\bar j)\in F$. Then by~\cite[Lemma~12.5]{Postnikov}, $\Delta_F$ is a simplex in $\R^{m+n}$ if and only if $F$ is a \emph{forest} of $G$ (i.e., a subset of edges of $G$ that contains no cycles). Moreover, the dimension of $\Delta_F$ is $|F|-1$, and thus $\Delta_F$ is top-dimensional (that is, $(m+n-2)$-dimensional) if and only if $F$ is a spanning tree of $G$.

\begin{definition}\label{dfn:intersect_properly}
We say that two simplices $\Delta_F$ and $\Delta_{F'}$ \emph{intersect by their common face} if $\Delta_F\cap\Delta_{F'}=\Delta_{F\cap F'}$. A \emph{triangulation} $\triang$ of $Q_G$ is a simplicial complex such that each simplex is of the form $\Delta_F$ for some forest $F\subset G$, any two simplices in $\triang$ intersect by their common face, and the union of these simplices is $Q_G$.
\end{definition}
It turns out that the above condition admits a simple combinatorial characterization:
\begin{definition}\label{dfn:compatible:intro}

We say that two forests $F,F'$ are \emph{compatible} if there does not exist a pair $M\subset F$, $M'\subset F'$ of partial matchings such that $M\neq M'$ but for all $i\in [m]$ and $\bar j\in[\bar n]$ we have $\deg_i(M)=\deg_i(M')$ and $\deg_{\bar j}(M)=\deg_{\bar j}(M')$.
\end{definition}
 Here $\deg_i(M):=|N_i(M)|$ and $\deg_{\bar j}(M):=|N_{\bar j}(M)|$  denote the \emph{degrees} of $i$ and $\bar j$ in $M$, and a \emph{partial matching} is a subgraph $M\subset G$ such that the degree of every vertex of $G$ in $M$ is at most $1$.
\begin{lemma}\label{lemma:compatible:intro}
Two simplices $\Delta_F$ and $\Delta_{F'}$ intersect by their common face if and only if  $F$ and $F'$ are compatible.
\end{lemma}
This lemma is a special case of Lemma~\ref{lemma:compatible}.

For a triangulation $\triang$ of $Q_G$, we denote by
\[\Forests(\triang):=\{F\subset G\mid \Delta_F\in\triang\}\]
the collection of \emph{forests} of $\triang$. Just as any other (pure) simplicial complex, $\triang$ is determined by its top-dimensional simplices, so we denote
\[\Trees(\triang):=\{T\in\Forests(\triang)\mid T\text{ is a spanning tree of $G$}\}. \]
Since all top-dimensional simplices $\Delta_T$ have the same volume (by~\cite[Lemma~12.5]{Postnikov}), it follows that a triangulation $\triang$ of $Q_G$ corresponds to a maximal by size collection $\Trees(\triang)$ of pairwise compatible spanning trees of $G$.

Given a spanning tree $T\subset G$, introduce the \emph{left-degree vector} $\LDm(T):=(d_1,\dots,d_m)$ given by $d_i:=\deg_i(T)-1$. We similarly define the \emph{right-degree vector} $\RDm(T):=(d_{\bar1},\dots,d_{\bar n})$.

\begin{lemma}[{\cite[Lemma~12.7]{Postnikov}}]\label{lemma:LD_bij}
Given a triangulation $\triang$ of $Q_G$, we have $\LDm(T)\neq \LDm(T')$ for $T\neq T'\in\Trees(\triang)$, and the set $\{\LDm(T)\mid T\in\Trees(\triang)\}$ equals $P_G^-\cap \Z^m$.
\end{lemma}

In other words, every integer point of $P_G^-$ appears as a left-degree vector for a unique tree in any triangulation. Thus a triangulation $\triang$ gives rise to a bijection $\phi_\triang:P_G^-\cap\Z^m\to P_{G^*}^-\cap Z^n$, where $G^*\subset K_{n,m}$ is the graph with edge set $\{(j,\bar i)\mid (i,\bar j)\in G\}$. In particular, when $G=K_{m,n}$, the map $\phi_\triang$ is a bijection between $(n-1)\Delta^{m-1}\cap\Z^m$ and $(m-1)\Delta^{n-1}\cap\Z^n$. Note that each of the two sets has cardinality ${n+m-2\choose n-1}$. 

\begin{figure}
  \hidefigure{
    \scalebox{0.95}{

}
}
\caption{\label{fig:trianguloid} A family $\Trees(\triang)$ of trees for a triangulation $\triang$ of $Q_{K_{m,n}}$ for $m=3, n=4$ (left), and the corresponding trianguloid $\svt:=\svt_\triang$ (right). The white (resp., black) vertices of $\svt$ are the lattice points of $P_G^-=(n-1)\Delta_{[m]}$ (resp., of $P_G=n\Delta_{[m]}$). Each point $b\in P_G^-$ corresponds to a tree $T_\svt(b)\in\Trees(\triang)$ with $\LDm(T_\svt(b))=b$ so that the outgoing arrows of $b$ in $\svt$ in the direction of $e_i$ are labeled by the neighbors of $i$ in $T_\svt(b)$.}
\end{figure}

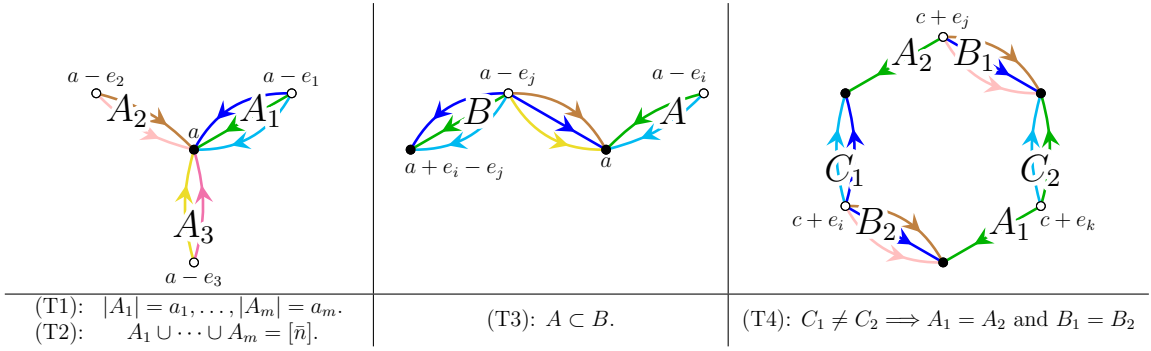
\begin{figure}

  \hidefigure{

\begin{tabular}{c|c|c}
\scalebox{1.0}{
\begin{tikzpicture}[baseline=(ZERO.base)]
\coordinate (ZERO) at (0:0.00);
\tikzset{myptr/.style={decoration={markings,mark=at position 0.7 with %
    {\arrow[scale=1.5,>=stealth]{>}}},postaction={decorate}}}
\coordinate (L3) at (90:3.00);
\coordinate (L2) at (330:3.00);
\coordinate (L1) at (210:3.00);
\coordinate (LL3) at (90:4.50);
\coordinate (LL2) at (330:4.50);
\coordinate (LL1) at (210:4.50);
\node[scale=0.3,draw,circle] (N0x1x1) at (barycentric cs:L1=0.0,L2=0.3333333333333333,L3=0.3333333333333333) { };
\node[scale=0.3,draw,circle] (N1x0x1) at (barycentric cs:L1=0.3333333333333333,L2=0.0,L3=0.3333333333333333) { };
\node[scale=0.3,draw,circle] (N1x1x0) at (barycentric cs:L1=0.3333333333333333,L2=0.3333333333333333,L3=0.0) { };
\node[scale=0.3,draw,circle,fill=black] (N1x1x1) at (barycentric cs:LL1=0.25,LL2=0.25,LL3=0.25) { };
\draw[draw=cyan!80!white,myptr,line width=1pt] (N0x1x1) to[bend right=-30] (N1x1x1);
\draw[draw=green!70!black,myptr,line width=1pt] (N0x1x1) to[bend right=0] (N1x1x1);
\draw[draw=blue,myptr,line width=1pt] (N0x1x1) to[bend right=30] (N1x1x1);
\draw[draw=none] (N0x1x1) to[] node[anchor=center, pos=0.3,scale=1.2] {\contour{white}{$A_1$}} (N1x1x1);
\draw[draw=brown,myptr,line width=1pt] (N1x0x1) to[bend right=-15] (N1x1x1);
\draw[draw=pink,myptr,line width=1pt] (N1x0x1) to[bend right=15] (N1x1x1);
\draw[draw=none] (N1x0x1) to[] node[anchor=center, pos=0.3,scale=1.2] {\contour{white}{$A_2$}} (N1x1x1);
\draw[draw=yellow!90!black,myptr,line width=1pt] (N1x1x0) to[bend right=-15] (N1x1x1);
\draw[draw=magenta!70,myptr,line width=1pt] (N1x1x0) to[bend right=15] (N1x1x1);
\draw[draw=none] (N1x1x0) to[] node[anchor=center, pos=0.3,scale=1.2] {\contour{white}{$A_3$}} (N1x1x1);
\node[scale=0.3,draw,circle] (N0x1x1) at (barycentric cs:L1=0.0,L2=0.3333333333333333,L3=0.3333333333333333) { };
\node[scale=0.3,draw,circle] (N1x0x1) at (barycentric cs:L1=0.3333333333333333,L2=0.0,L3=0.3333333333333333) { };
\node[scale=0.3,draw,circle] (N1x1x0) at (barycentric cs:L1=0.3333333333333333,L2=0.3333333333333333,L3=0.0) { };
\node[scale=0.3,draw,circle,fill=black] (N1x1x1) at (barycentric cs:LL1=0.25,LL2=0.25,LL3=0.25) { };
\node[anchor=south,scale=0.7] (BLAH) at (N1x1x1.north) {$a$};
\node[anchor=-90,scale=0.7] (BLAH) at (N0x1x1.center) {$a-e_1$};
\node[anchor=-90,scale=0.7] (BLAH) at (N1x0x1.center) {$a-e_2$};
\node[anchor=90,scale=0.7] (BLAH) at (N1x1x0.center) {$a-e_3$};
\end{tikzpicture}}
&
\scalebox{1.0}{
\begin{tikzpicture}[baseline=(ZERO.base)]
\coordinate (ZERO) at (0:0.00);
\tikzset{myptr/.style={decoration={markings,mark=at position 0.7 with %
    {\arrow[scale=1.5,>=stealth]{>}}},postaction={decorate}}}
\coordinate (L3) at (90:3.00);
\coordinate (L2) at (330:3.00);
\coordinate (L1) at (210:3.00);
\coordinate (LL3) at (90:4.50);
\coordinate (LL2) at (330:4.50);
\coordinate (LL1) at (210:4.50);
\node[scale=0.3,draw,circle] (N0x1x1) at (barycentric cs:L1=0.0,L2=0.3333333333333333,L3=0.3333333333333333) { };
\node[scale=0.3,draw,circle] (N1x0x1) at (barycentric cs:L1=0.3333333333333333,L2=0.0,L3=0.3333333333333333) { };
\node[scale=0.3,draw,circle,fill=black] (N1x1x1) at (barycentric cs:LL1=0.25,LL2=0.25,LL3=0.25) { };
\node[scale=0.3,draw,circle,fill=black] (N2x0x1) at (barycentric cs:LL1=0.5,LL2=0.0,LL3=0.25) { };
\draw[draw=cyan!80!white,myptr,line width=1pt] (N0x1x1) to[bend right=-15] (N1x1x1);
\draw[draw=green!70!black,myptr,line width=1pt] (N0x1x1) to[bend right=15] (N1x1x1);
\draw[draw=none] (N0x1x1) to[] node[anchor=center, pos=0.3,scale=1.2] {\contour{white}{$A$}} (N1x1x1);
\draw[draw=brown,myptr,line width=1pt] (N1x0x1) to[bend right=-30] (N1x1x1);
\draw[draw=blue,myptr,line width=1pt] (N1x0x1) to[bend right=0] (N1x1x1);
\draw[draw=yellow!90!black,myptr,line width=1pt] (N1x0x1) to[bend right=30] (N1x1x1);
\draw[draw=none] (N1x0x1) to[] node[anchor=center, pos=0.3,scale=1.2] {\contour{white}{$$}} (N1x1x1);
\draw[draw=cyan!80!white,myptr,line width=1pt] (N1x0x1) to[bend right=-30] (N2x0x1);
\draw[draw=green!70!black,myptr,line width=1pt] (N1x0x1) to[bend right=0] (N2x0x1);
\draw[draw=blue,myptr,line width=1pt] (N1x0x1) to[bend right=30] (N2x0x1);
\draw[draw=none] (N1x0x1) to[] node[anchor=center, pos=0.3,scale=1.2] {\contour{white}{$B$}} (N2x0x1);
\node[scale=0.3,draw,circle] (N0x1x1) at (barycentric cs:L1=0.0,L2=0.3333333333333333,L3=0.3333333333333333) { };
\node[scale=0.3,draw,circle] (N1x0x1) at (barycentric cs:L1=0.3333333333333333,L2=0.0,L3=0.3333333333333333) { };
\node[scale=0.3,draw,circle,fill=black] (N1x1x1) at (barycentric cs:LL1=0.25,LL2=0.25,LL3=0.25) { };
\node[scale=0.3,draw,circle,fill=black] (N2x0x1) at (barycentric cs:LL1=0.5,LL2=0.0,LL3=0.25) { };
\node[anchor=north,scale=0.7] (BLAH) at (N1x1x1.center) {$a$};
\node[anchor=160,scale=0.7] (BLAH) at (N2x0x1.south west) {$a+e_i-e_j$};
\node[anchor=-30,scale=0.7] (BLAH) at (N0x1x1.north east) {$a-e_i$};
\node[anchor=south,scale=0.7] (BLAH) at (N1x0x1.center) {$a-e_j$};
\end{tikzpicture}}
&
\scalebox{1.0}{
\begin{tikzpicture}[baseline=(ZERO.base)]
\coordinate (ZERO) at (0:0.00);
\tikzset{myptr/.style={decoration={markings,mark=at position 0.7 with %
    {\arrow[scale=1.5,>=stealth]{>}}},postaction={decorate}}}
\coordinate (L3) at (90:1.50);
\coordinate (L2) at (330:1.50);
\coordinate (L1) at (210:1.50);
\coordinate (LL3) at (90:3.00);
\coordinate (LL2) at (330:3.00);
\coordinate (LL1) at (210:3.00);
\node[scale=0.3,draw,circle] (N0x0x1) at (barycentric cs:L1=0.0,L2=0.0,L3=0.5) { };
\node[scale=0.3,draw,circle] (N0x1x0) at (barycentric cs:L1=0.0,L2=0.5,L3=0.0) { };
\node[scale=0.3,draw,circle] (N1x0x0) at (barycentric cs:L1=0.5,L2=0.0,L3=0.0) { };
\node[scale=0.3,draw,circle,fill=black] (N0x1x1) at (barycentric cs:LL1=0.0,LL2=0.3333333333333333,LL3=0.3333333333333333) { };
\node[scale=0.3,draw,circle,fill=black] (N1x0x1) at (barycentric cs:LL1=0.3333333333333333,LL2=0.0,LL3=0.3333333333333333) { };
\node[scale=0.3,draw,circle,fill=black] (N1x1x0) at (barycentric cs:LL1=0.3333333333333333,LL2=0.3333333333333333,LL3=0.0) { };
\draw[draw=brown,myptr,line width=1pt] (N0x0x1) to[bend right=-30] (N0x1x1);
\draw[draw=blue,myptr,line width=1pt] (N0x0x1) to[bend right=0] (N0x1x1);
\draw[draw=pink,myptr,line width=1pt] (N0x0x1) to[bend right=30] (N0x1x1);
\draw[draw=none] (N0x0x1) to[] node[anchor=center, pos=0.3,scale=1.2] {\contour{white}{$B_1$}} (N0x1x1);
\draw[draw=cyan!80!white,myptr,line width=1pt] (N0x1x0) to[bend right=-15] (N0x1x1);
\draw[draw=green!70!black,myptr,line width=1pt] (N0x1x0) to[bend right=15] (N0x1x1);
\draw[draw=none] (N0x1x0) to[] node[anchor=center, pos=0.3,scale=1.2] {\contour{white}{$C_2$}} (N0x1x1);
\draw[draw=green!70!black,myptr,line width=1pt] (N0x0x1) to[bend right=0] (N1x0x1);
\draw[draw=none] (N0x0x1) to[] node[anchor=center, pos=0.3,scale=1.2] {\contour{white}{$A_2$}} (N1x0x1);
\draw[draw=cyan!80!white,myptr,line width=1pt] (N1x0x0) to[bend right=-15] (N1x0x1);
\draw[draw=blue,myptr,line width=1pt] (N1x0x0) to[bend right=15] (N1x0x1);
\draw[draw=none] (N1x0x0) to[] node[anchor=center, pos=0.3,scale=1.2] {\contour{white}{$C_1$}} (N1x0x1);
\draw[draw=green!70!black,myptr,line width=1pt] (N0x1x0) to[bend right=0] (N1x1x0);
\draw[draw=none] (N0x1x0) to[] node[anchor=center, pos=0.3,scale=1.2] {\contour{white}{$A_1$}} (N1x1x0);
\draw[draw=brown,myptr,line width=1pt] (N1x0x0) to[bend right=-30] (N1x1x0);
\draw[draw=blue,myptr,line width=1pt] (N1x0x0) to[bend right=0] (N1x1x0);
\draw[draw=pink,myptr,line width=1pt] (N1x0x0) to[bend right=30] (N1x1x0);
\draw[draw=none] (N1x0x0) to[] node[anchor=center, pos=0.3,scale=1.2] {\contour{white}{$B_2$}} (N1x1x0);
\node[scale=0.3,draw,circle] (N0x0x1) at (barycentric cs:L1=0.0,L2=0.0,L3=0.5) { };
\node[scale=0.3,draw,circle] (N0x1x0) at (barycentric cs:L1=0.0,L2=0.5,L3=0.0) { };
\node[scale=0.3,draw,circle] (N1x0x0) at (barycentric cs:L1=0.5,L2=0.0,L3=0.0) { };
\node[scale=0.3,draw,circle,fill=black] (N0x1x1) at (barycentric cs:LL1=0.0,LL2=0.3333333333333333,LL3=0.3333333333333333) { };
\node[scale=0.3,draw,circle,fill=black] (N1x0x1) at (barycentric cs:LL1=0.3333333333333333,LL2=0.0,LL3=0.3333333333333333) { };
\node[scale=0.3,draw,circle,fill=black] (N1x1x0) at (barycentric cs:LL1=0.3333333333333333,LL2=0.3333333333333333,LL3=0.0) { };
\node[anchor=30,scale=0.7] (BLAH) at (N1x0x0.center) {$c+e_i$};
\node[anchor=south,scale=0.7] (BLAH) at (N0x0x1.center) {$c+e_j$};
\node[anchor=150,scale=0.7] (BLAH) at (N0x1x0.center) {$c+e_k$};
\end{tikzpicture}}
\\\hline

\scalebox{0.7}{
\begin{tabular}{cc}
\axCardinalityK:
&
$|A_1|=a_1,\dots,|A_m|=a_m.$
\\

\axExistenceK:
&
$A_1\cup \dots\cup A_m=[\bar n]$.
\\

\end{tabular}
}
&
\scalebox{0.7}{\axContainmentK: $A\subset B$.}
&
\scalebox{0.7}{\axHexagonK: $C_1\neq C_2\Longrightarrow A_1=A_2$ and $B_1=B_2$}
\\

\end{tabular}
  }
  \caption{\label{fig:axioms}Axioms for trianguloids.}
\end{figure}

\section{Main results: the case $G=K_{m,n}$}\label{sec:main_K}

We concentrate on  characterizing triangulations by a set of axioms. For simplicity, we first state our definitions and results in the case when $G$ is the complete bipartite graph $K_{m,n}$. For the rest of this section, we assume $G=K_{m,n}$. We denote $\DZ(m,k):=k\Delta_{[m]}\cap\Z^m$. Thus we have $P_G\cap\Z^m=\DZ(m,n)$ and $P_G^-\cap\Z^m=\DZ(m,n-1)$.

\def\SVTG{\Gamma_G} 

Let us consider a directed graph $\SVT$ with vertex set $V(\SVT)=\DZ(m,n-1)\sqcup \DZ(m,n)$ and edge set $\SVTE:=\{b\to b+e_i\mid b\in \DZ(m,n-1), i\in[m]\}$. We alternatively denote an edge $b\to a$, where $a=b+e_i$ for some $i\in [m]$, by either $b\toidot$ or $\dottoi a$.

\begin{definition}\label{dfn:pre_trianguloid_K}
A \emph{pre-trianguloid} is a map $\svt: \SVTE\to 2^{[\bar n]}$ satisfying the following axioms:
  \begin{enumerate}
  \item[(T1)]\label{ax:cardinality_K} for every edge $(\dottoi a)\in\SVTE$, we have $|\svt(\dottoi a)|=a_i$.

  \item[(T2)]\label{ax:existence_K} for each $a\in \DZ(m,n)$ and $\bar j\in [\bar n]$, there exists an index $i\in [m]$ such that $\bar j\in \svt(\dottoi a )$.
  \item[(T3)]\label{ax:containment_K} If both $a$ and $a':=a+e_i-e_j$ belong to $\DZ(m,n)$ then we have
    $$\svt(\dottoi a)\subset \svt(\dottoi a').$$ 
  \end{enumerate}
\end{definition}

These axioms are illustrated in Figure~\ref{fig:axioms}.

\begin{remark}\label{rmk:set_partitionK}
  By Axiom~\axCardinalityK, we have $\sum_{i\in[m]} |\svt(\dottoi a)|=n$, and by Axiom~\axExistenceK, the union of these sets is $[\bar n]$. Thus these sets are pairwise disjoint, so the index $i\in[m]$ in Axiom~\axExistenceK not only exists but is also unique.
\end{remark}

Consider a triangulation $\triang$ of $Q_{K_{m,n}}=\DD$. By Lemma~\ref{lemma:LD_bij}, for each $b\in \DZ(m,n-1)$, there is a unique tree $T_\triang(b)\in\Trees(\triang)$ such that $\LDm(T_\triang(b))=b$. Define a map $\svt_\triang:\SVTE\to 2^{[\bar n]}$ by
\begin{equation}\label{eq:svt_triang_K}
\svt_\triang(b\toidot):=\{\bar j\mid (i,\bar j)\in T_\triang(b)\}.
\end{equation}

\begin{proposition}\label{prop:triang_implies_svt_K}
If $\triang$ is a triangulation of $\DD$ then $\svt_\triang$ is a pre-trianguloid.
\end{proposition}

See Figure~\ref{fig:trianguloid} for an example.


Conversely, given a pre-trianguloid $\svt:\SVTE\to 2^{[\bar n]}$ and a point $b\in \DZ(m,n-1)$, one can define a subgraph $T_\svt(b)\subset K_{m,n}$ with edge set
\begin{equation}\label{eq:T_svt}
T_\svt(b):=\{(i,\bar j)\mid \bar j\in \svt(b\toidot)\}.
\end{equation}
 
\begin{proposition}\label{prop:svt_implies_trees_K}
For a pre-trianguloid $\svt$ and a point $b\in \DZ(m,n-1)$, $T_\svt(b)$ is a spanning tree of $K_{m,n}$.
\end{proposition}

Thus $\{\Delta_{T_\svt(b)}\mid b\in \DZ(m,n-1)\}$ is a collection of full-dimensional simplices whose total volume is equal to the volume of $\DD$.  However, it may happen that these simplices do not in fact form a triangulation of $\DD$, see Figure~\ref{fig:bad} for an example.

\begin{figure}

\hidefigure{
\scalebox{1.3}{

\begin{tabular}{cc}
\begin{tikzpicture}
\coordinate (L3) at (90:1.00);
\coordinate (L2) at (330:1.00);
\coordinate (L1) at (210:1.00);
\coordinate (LL3) at (90:2.00);
\coordinate (LL2) at (330:2.00);
\coordinate (LL1) at (210:2.00);
\node[scale=0.4,draw,rectangle,anchor=center] (N0x0x1) at (barycentric cs:L1=0.0,L2=0.0,L3=0.5) {\begin{tikzpicture}
\node[fill=black,circle,label=left:$1$] (a0) at (0.00,0.00) {};
\node[fill=black,circle,label=left:$2$] (a1) at (0.00,0.90) {};
\node[fill=black,circle,label=left:$3$] (a2) at (0.00,1.80) {};
\node[fill=red,circle,label=right:$\bar{1}$] (b0) at (1.80,0.00) {};
\node[fill=cyan!80!white,circle,label=right:$\bar{2}$] (b1) at (1.80,1.80) {};
\draw[line width=2.0pt,draw=cyan!80!white] (a0) -- (b1);
\draw[line width=2.0pt,draw=red] (a1) -- (b0);
\draw[line width=2.0pt,draw=red] (a2) -- (b0);
\draw[line width=2.0pt,draw=cyan!80!white] (a2) -- (b1);
\end{tikzpicture}};
\node[scale=0.4,draw,rectangle,anchor=center] (N0x1x0) at (barycentric cs:L1=0.0,L2=0.5,L3=0.0) {\begin{tikzpicture}
\node[fill=black,circle,label=left:$1$] (a0) at (0.00,0.00) {};
\node[fill=black,circle,label=left:$2$] (a1) at (0.00,0.90) {};
\node[fill=black,circle,label=left:$3$] (a2) at (0.00,1.80) {};
\node[fill=red,circle,label=right:$\bar{1}$] (b0) at (1.80,0.00) {};
\node[fill=cyan!80!white,circle,label=right:$\bar{2}$] (b1) at (1.80,1.80) {};
\draw[line width=2.0pt,draw=red] (a0) -- (b0);
\draw[line width=2.0pt,draw=red] (a1) -- (b0);
\draw[line width=2.0pt,draw=cyan!80!white] (a1) -- (b1);
\draw[line width=2.0pt,draw=cyan!80!white] (a2) -- (b1);
\end{tikzpicture}};
\node[scale=0.4,draw,rectangle,anchor=center] (N1x0x0) at (barycentric cs:L1=0.5,L2=0.0,L3=0.0) {\begin{tikzpicture}
\node[fill=black,circle,label=left:$1$] (a0) at (0.00,0.00) {};
\node[fill=black,circle,label=left:$2$] (a1) at (0.00,0.90) {};
\node[fill=black,circle,label=left:$3$] (a2) at (0.00,1.80) {};
\node[fill=red,circle,label=right:$\bar{1}$] (b0) at (1.80,0.00) {};
\node[fill=cyan!80!white,circle,label=right:$\bar{2}$] (b1) at (1.80,1.80) {};
\draw[line width=2.0pt,draw=red] (a0) -- (b0);
\draw[line width=2.0pt,draw=cyan!80!white] (a0) -- (b1);
\draw[line width=2.0pt,draw=cyan!80!white] (a1) -- (b1);
\draw[line width=2.0pt,draw=red] (a2) -- (b0);
\end{tikzpicture}};
\end{tikzpicture}
&
\scalebox{1.0}{
\begin{tikzpicture}
\tikzset{myptr/.style={decoration={markings,mark=at position 0.7 with %
    {\arrow[scale=1.5,>=stealth]{>}}},postaction={decorate}}}
\coordinate (L3) at (90:1.00);
\coordinate (L2) at (330:1.00);
\coordinate (L1) at (210:1.00);
\coordinate (LL3) at (90:2.00);
\coordinate (LL2) at (330:2.00);
\coordinate (LL1) at (210:2.00);
\node[scale=0.3,draw,circle] (N0x0x1) at (barycentric cs:L1=0.0,L2=0.0,L3=0.5) { };
\node[scale=0.3,draw,circle] (N0x1x0) at (barycentric cs:L1=0.0,L2=0.5,L3=0.0) { };
\node[scale=0.3,draw,circle] (N1x0x0) at (barycentric cs:L1=0.5,L2=0.0,L3=0.0) { };
\node[scale=0.3,draw,circle,fill=black] (N0x1x1) at (barycentric cs:LL1=0.0,LL2=0.3333333333333333,LL3=0.3333333333333333) { };
\node[scale=0.3,draw,circle,fill=black] (N1x0x1) at (barycentric cs:LL1=0.3333333333333333,LL2=0.0,LL3=0.3333333333333333) { };
\node[scale=0.3,draw,circle,fill=black] (N1x1x0) at (barycentric cs:LL1=0.3333333333333333,LL2=0.3333333333333333,LL3=0.0) { };
\node[scale=0.3,draw,circle,fill=black] (N0x0x2) at (barycentric cs:LL1=0.0,LL2=0.0,LL3=0.6666666666666666) { };
\node[scale=0.3,draw,circle,fill=black] (N0x2x0) at (barycentric cs:LL1=0.0,LL2=0.6666666666666666,LL3=0.0) { };
\node[scale=0.3,draw,circle,fill=black] (N2x0x0) at (barycentric cs:LL1=0.6666666666666666,LL2=0.0,LL3=0.0) { };
\draw[draw=red,myptr,line width=1pt] (N0x0x1) to[bend right=0] node[anchor=center, pos=0.3,scale=0.4] {\contour{white}{$\bar 1$}} (N0x1x1);
\draw[draw=cyan!80!white,myptr,line width=1pt] (N0x1x0) to[bend right=0] node[anchor=center, pos=0.3,scale=0.4] {\contour{white}{$\bar 2$}} (N0x1x1);
\draw[draw=cyan!80!white,myptr,line width=1pt] (N0x0x1) to[bend right=0] node[anchor=center, pos=0.3,scale=0.4] {\contour{white}{$\bar 2$}} (N1x0x1);
\draw[draw=red,myptr,line width=1pt] (N1x0x0) to[bend right=0] node[anchor=center, pos=0.3,scale=0.4] {\contour{white}{$\bar 1$}} (N1x0x1);
\draw[draw=red,myptr,line width=1pt] (N0x1x0) to[bend right=0] node[anchor=center, pos=0.3,scale=0.4] {\contour{white}{$\bar 1$}} (N1x1x0);
\draw[draw=cyan!80!white,myptr,line width=1pt] (N1x0x0) to[bend right=0] node[anchor=center, pos=0.3,scale=0.4] {\contour{white}{$\bar 2$}} (N1x1x0);
\draw[draw=red,myptr,line width=1pt] (N0x0x1) to[bend right=-15] node[anchor=center, pos=0.3,scale=0.4] {\contour{white}{$\bar 1$}} (N0x0x2);
\draw[draw=cyan!80!white,myptr,line width=1pt] (N0x0x1) to[bend right=15] node[anchor=center, pos=0.3,scale=0.4] {\contour{white}{$\bar 2$}} (N0x0x2);
\draw[draw=red,myptr,line width=1pt] (N0x1x0) to[bend right=-15] node[anchor=center, pos=0.3,scale=0.4] {\contour{white}{$\bar 1$}} (N0x2x0);
\draw[draw=cyan!80!white,myptr,line width=1pt] (N0x1x0) to[bend right=15] node[anchor=center, pos=0.3,scale=0.4] {\contour{white}{$\bar 2$}} (N0x2x0);
\draw[draw=red,myptr,line width=1pt] (N1x0x0) to[bend right=-15] node[anchor=center, pos=0.3,scale=0.4] {\contour{white}{$\bar 1$}} (N2x0x0);
\draw[draw=cyan!80!white,myptr,line width=1pt] (N1x0x0) to[bend right=15] node[anchor=center, pos=0.3,scale=0.4] {\contour{white}{$\bar 2$}} (N2x0x0);
\node[scale=0.3,draw,circle] (N0x0x1) at (barycentric cs:L1=0.0,L2=0.0,L3=0.5) { };
\node[scale=0.3,draw,circle] (N0x1x0) at (barycentric cs:L1=0.0,L2=0.5,L3=0.0) { };
\node[scale=0.3,draw,circle] (N1x0x0) at (barycentric cs:L1=0.5,L2=0.0,L3=0.0) { };
\node[scale=0.3,draw,circle,fill=black] (N0x1x1) at (barycentric cs:LL1=0.0,LL2=0.3333333333333333,LL3=0.3333333333333333) { };
\node[scale=0.3,draw,circle,fill=black] (N1x0x1) at (barycentric cs:LL1=0.3333333333333333,LL2=0.0,LL3=0.3333333333333333) { };
\node[scale=0.3,draw,circle,fill=black] (N1x1x0) at (barycentric cs:LL1=0.3333333333333333,LL2=0.3333333333333333,LL3=0.0) { };
\node[scale=0.3,draw,circle,fill=black] (N0x0x2) at (barycentric cs:LL1=0.0,LL2=0.0,LL3=0.6666666666666666) { };
\node[scale=0.3,draw,circle,fill=black] (N0x2x0) at (barycentric cs:LL1=0.0,LL2=0.6666666666666666,LL3=0.0) { };
\node[scale=0.3,draw,circle,fill=black] (N2x0x0) at (barycentric cs:LL1=0.6666666666666666,LL2=0.0,LL3=0.0) { };
\node[anchor=north east,scale=0.6] (BASIS) at (3.50,2.20) {\scalebox{0.8}{
\begin{tikzpicture}
\coordinate (e0) at (0.00,0.00);
\node[anchor=30] (ee1) at (-150:1) {$e_1$};
\node[anchor=150] (ee2) at (-30:1) {$e_2$};
\node[anchor=-90] (ee3) at (90:1) {$e_3$};
\draw[->] (e0) -- (ee1);
\draw[->] (e0) -- (ee2);
\draw[->] (e0) -- (ee3);
\end{tikzpicture}}};
\end{tikzpicture}}
\\

\end{tabular}

}
}
\caption{\label{fig:bad} A pre-trianguloid for $m=3,n=2$ which does not correspond to a triangulation of $\DD$. The trees on the left are pairwise non-compatible.}
\end{figure}
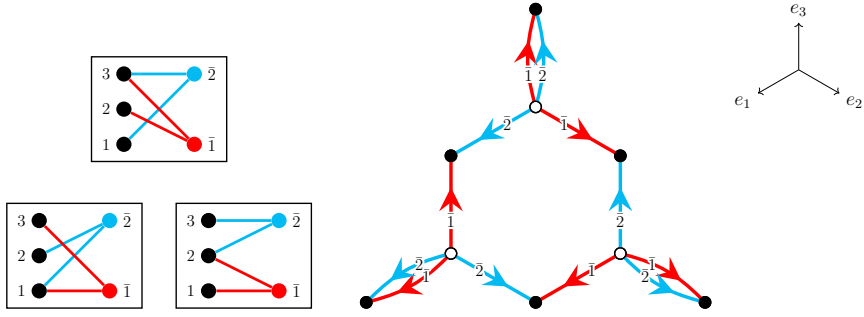

We fix this by introducing an additional axiom. 
\begin{definition}
  A \emph{trianguloid} is a pre-trianguloid $\svt: \SVTE\to 2^{[\bar n]}$ satisfying the following \emph{Hexagon axiom}:
\begin{enumerate}
\item[(T4)] \label{ax:hexagon_K} let  $c\in \DZ(m,n-2)$ and consider three distinct indices $i,j,k\in [m]$ such that $\svt( c+e_i\tojdot)\neq \svt( c+e_k\tojdot)$. Then we have
 \[\svt( c+e_i\tokdot)=\svt(c+e_j\tokdot )\quad \text{and}\quad \svt( c+e_j\toidot )=\svt( c+e_k\toidot ).\]
\end{enumerate}
\end{definition}

Axioms~\axCardinalityK-\axHexagonK are illustrated in Figure~\ref{fig:axioms}. 

The following is our main result for the case $G=K_{m,n}$.
\begin{theorem}\label{thm:svt_implies_triang_K}
  The map $\triang\mapsto\svt_\triang$ is a bijection between triangulations of $\DD$ and trianguloids.
\end{theorem}
The generalization of this to arbitrary $G$ is given in Theorem~\ref{thm:svt_implies_triang}.

We now describe a compact way of encoding a pre-trianguloid. Introduce another directed graph $\ECT$ with vertex set $V(\ECT):=\DZ(m,n-1)$ and edge set
\[\ECTE:=\{b\to b'\mid b,b'\in \DZ(m,n-1),\quad  b'=b+e_i-e_j\quad \text{for some $i\neq j\in[m]$}\}.\]

Consider an edge $b\to b'\in\ECTE$. Then by Axioms~\axCardinalityK and~\axContainmentK, there is a single index $\bar j$ such that $\svt( b'\toidot )=\svt(  b\toidot)\sqcup\{\bar j\}$ (disjoint union). Thus each pre-trianguloid $\svt$ defines an edge coloring $\ect_\svt:\ECTE\to [\bar n]$ sending an edge $(b\to b')\in \ECTE$ to the above index $\bar j$. It is easy to see that a pre-trianguloid can in fact be uniquely reconstructed from this map. An example is given in Figure~\ref{fig:ect}.

\begin{figure}
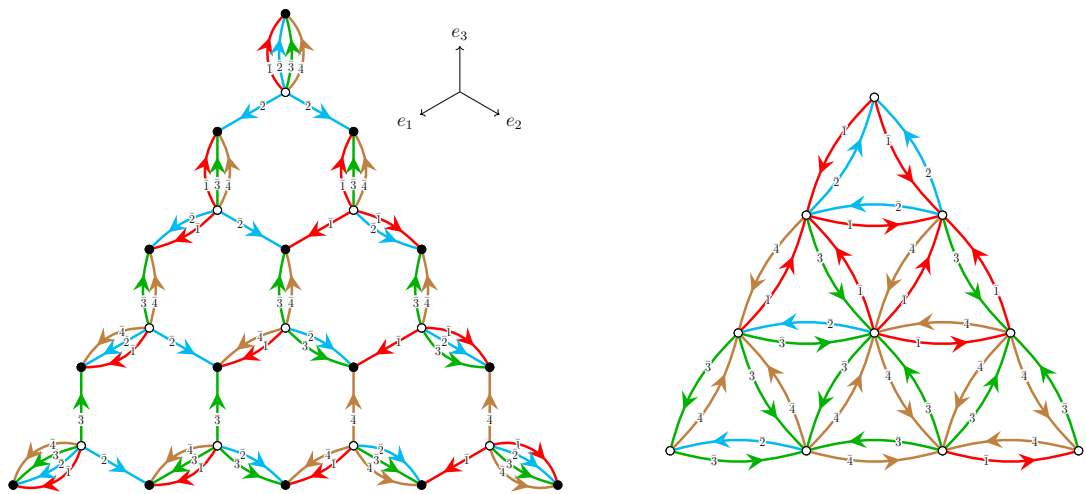


  \hidefigure{
    
    \scalebox{0.95}{

}
}
\caption{\label{fig:ect} A trianguloid $\svt$ (left) and the corresponding edge coloring $\ect_\svt$ (right). Here $m=3$ and $n=4$.}
\end{figure}


We finish by going back to our original question, deducing an analog of~\cite[Theorem~5]{BZ} as a simple consequence of the above results.

\begin{theorem}\label{thm:bijections_different_K}
  For two different triangulations $\triang,\triang'$ of $Q_{K_{m,n}}$, the maps $\phi_{\triang}, \phi_{\triang'}$ are different as well.
\end{theorem}

\section{Motivation}\label{sec:history}
Before stating our main results for the case of arbitrary connected $G\subset K_{m,n}$, we discuss (very informally) some of the objects corresponding to triangulations of $\Delta_{[m]}\times\Delta_{[\bar n]}$ that have been studied earlier. The main goal of this section is to provide intuition and motivating examples; the only two things that we actually use in the remainder of the paper are Definition~\ref{dfn:RSM} and part~\eqref{item:bij:LDp} of Lemma~\ref{lemma:LDp_bij}.

\def\LeftSupport{\operatorname{I}}
\def\RightSupport{\operatorname{J}}
\subsection{Forests, matchings, and tropical oriented matroids}\label{sec:forests-match-trop}
In this section, we list several ways to describe a triangulation $\triang$ of $Q_G$, where $G\subset K_{m,n}$ is an arbitrary connected graph. Recall that $\Forests(\triang)$ and $\Trees(\triang)$ denotes the collection of \emph{forests} and \emph{trees} of $\triang$ respectively.
\begin{definition}\label{dfn:RSM}
We say that a forest $F\subset G$ is a \emph{right semi-matching} if $\deg_{\bar j}(F)=1$ for all $\bar j\in[\bar n]$, and denote
\[\RM(\triang)=\{F\in\Forests(\triang)\mid F\text{ is a right semi-matching}\}.\]
\end{definition}
We similarly define \emph{left semi-matchings} to be forests $F\subset G$ such that $\deg_i(F)=1$ for all $i\in [m]$, and denote by $\LM(\triang)$ the set of left semi-matchings in $\Forests(\triang)$. Finally, recall that a forest $F$ is a \emph{partial matching} if $0\leq \deg_i(F),\deg_{\bar j}(F)\leq 1$ for all $i\in [m]$ and $\bar j\in[\bar n]$. In this case, we call the set $\LeftSupport(F):=\{i\in[m]\mid \deg_i(F)=1\}$ (resp., $\RightSupport(F):=\{\bar j\in[\bar n]\mid \deg_{\bar j}(F)=1\}$) the \emph{left support} (resp., the \emph{right support}) of $F$, and say that $F$ is a \emph{matching between $\LeftSupport(F)$ and $\RightSupport(F)$}.

We denote
\[\Matchings(\triang)=\{F\in\Forests(\triang)\mid F\text{ is a partial matching}\}.\]

The following result (cf. Figure~\ref{fig:matchings}) will follow as a simple corollary to Lemma~\ref{lemma:compatible}.
\begin{proposition}\label{prop:determined}
  A triangulation $\triang$ of $Q_G$ is determined uniquely by each of the following sets:
  \begin{itemize}
  \item $\Trees(\triang)$;
  \item $\RM(\triang)$;
  \item $\LM(\triang)$;
  \item $\Matchings(\triang)$.  
  \end{itemize}
  More precisely, for each of the four collections above, $\Forests(\triang)$ is equal to the set of all forests $F\subset G$ compatible (see Definition~\ref{dfn:compatible:intro}) with all $F'$ belonging to that collection.
\end{proposition}
\begin{proof}
Clearly $\triang$ is determined by $\Trees(\triang)$, and  $\Trees(\triang)$ determines $\RM(\triang)$. Let us show that $\RM(\triang)$ determines $\Matchings(\triang)$. Each partial matching $M\in\Matchings(\triang)$ is contained inside some tree $T\in\Trees(\triang)$ because $\triang$ is a pure simplicial complex, and then it is easy to see that there exists a right semi-matching $F\subset T$ such that $M\subset F$. This implies that $F\in\RM(\triang)$. Thus $\Matchings(\triang)$ is the set of all partial matchings of $G$ that are contained in some element of $\RM(\triang)$, so $\RM(\triang)$ determines $\Matchings(\triang)$. The proof that $\Trees(\triang)$ determines $\LM(\triang)$ which determines $\Matchings(\triang)$ is completely analogous. It suffices to show that $\Matchings(\triang)$ determines $\triang$. Explicitly, $\Forests(\triang)$ is the collection of all forests $F\subset G$ that do not contain a partial matching that is not in $\Matchings(\triang)$. This fact follows from Lemma~\ref{lemma:compatible:intro} (whose proof we defer to Section~\ref{sec:triang_implies_svt}).
\end{proof}

We now review the relationship between the above objects and \emph{tropical oriented matroids} of~\cite{AD}. It was conjectured in~\cite[Conjecture~5.1]{AD} that tropical oriented matroids are in bijection with subdivisions of $\DD$. Oh and Yoo~\cite{OYtom} proved that \emph{generic} tropical oriented matroids are in bijection with triangulations of $\DD$, and the case of general subdivisions was completed by Horn~\cite{Horn}.

A tropical oriented matroid $M$ is by definition a collection of \emph{types} satisfying some axioms, see~\cite{AD}. In the language of triangulations, types correspond to forests $F\in\Forests(\triang)$ such that $\deg_{\bar j}(F)\geq 1$ for all $\bar j\in [\bar n]$. Alternatively, a tropical oriented matroid is determined by the collection of its \emph{topes} or by the collection of its \emph{vertices}, see~\cite[Theorems~4.4 and~4.6]{AD}. The topes correspond to the elements of $\RM(\triang)$ and the vertices correspond to the elements of $\Trees(\triang)$, i.e., the right semi-matchings and the trees of $\triang$, respectively. Thus in the case of $G=K_{m,n}$, Proposition~\ref{prop:determined} follows from~\cite[Theorems~4.4 and~4.6]{AD} together with~\cite[Lemma~4.5]{OYmatching}.

\begin{figure}
  \def\nodesc{1}
 
  \begin{tikzpicture}[xscale=0.2,yscale=1.0]
    \node[scale=\nodesc,draw,anchor=south] (T) at (0,2) {$\Trees(\triang)$};
    \node[scale=\nodesc,draw,anchor=east] (R) at (-1,1) {$\RM(\triang)$};
    \node[scale=\nodesc,draw,anchor=west] (L) at (1,1) {$\LM(\triang)$};
    \node[scale=\nodesc,draw,anchor=north] (M) at (0,0) {$\Matchings(\triang)$};
    \draw[line width=0.5pt] (T)--(R)--(M)--(L)--(T);
  \end{tikzpicture}
  \caption{\label{fig:matchings}Different collections of forests of $\triang$ that determine it.}
\end{figure}
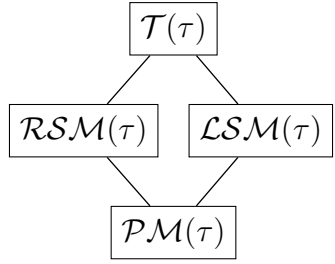

For a forest $F\subset G$, define
\[\LDp(F):=(\deg_i(F))_{i\in [m]}\in\Z^m,\quad \RDp(F):=(\deg_{\bar j}(F))_{\bar j\in [\bar n]}\in\Z^{n}.\]

\def\LRS{\operatorname{IJ}}
\def\MP{\operatorname{MSP}}

Let us denote by $\Matchings(G)$ the set of all partial matchings $F$ such that $F\subset G$. We define $\LRS_G\subset 2^{[m]}\times 2^{[\bar n]}$ by
\[\LRS_G:=\{(\LeftSupport(F),\RightSupport(F))\mid F\in\Matchings(G)\}.\]
Identifying a pair $(I,J)\in  2^{[m]}\times 2^{[\bar n]}$ with a vector $e_I+e_J:=\sum_{i\in I} e_i+\sum_{\bar j\in J} e_{\bar j}\in\R^{m+n}$, we see that $\LRS_G$ is the set of lattice points of a certain polytope in $\R^{m+n}$ which we call the \emph{matching support polytope} $\MP_G$ of $G$:
\[\MP_G:=\Conv \left(\{e_I+e_J\mid (I,J)\in\LRS_G\}\right)=\Conv \left(\{(\LDp(F),\RDp(F))\mid F\in\Matchings(G) \}\right).\]

We prove the following generalization of Lemma~\ref{lemma:LD_bij}.
\begin{lemma}\label{lemma:LDp_bij}
  Given a triangulation $\triang$ of $Q_G$, the following maps are bijections:
  \begin{enumerate}[\normalfont(1)]
  \item\label{item:bij:LD} $\LDm:\Trees(\triang)\to P_G^-\cap\Z^m$;
  \item\label{item:bij:RD} $\RDm:\Trees(\triang)\to P_{G^\ast}^-\cap\Z^n$;
  \item\label{item:bij:LDp} $\LDp:\RM(\triang)\to P_G\cap\Z^m$;
  \item\label{item:bij:RDp} $\RDp:\LM(\triang)\to P_{G^\ast}\cap\Z^n$;
  \item\label{item:bij:LDpRDp} $(\LDp,\RDp):\Matchings(\triang)\to \MP_G\cap\Z^{m+n}$.
  \end{enumerate}
\end{lemma}

\subsection{Newton polytopes and products of minors}\label{sec:newt-polyt-prod}
Fix a connected graph $G\subset K_{m,n}$ and consider an $m\times n$ matrix $f_G=(f_{ij})$ with $f_{ij}$ being an indeterminate for $(i,\bar j)\in G$ and $f_{ij}=0$ otherwise.

\def\Newton{\operatorname{Newton}}
For two subsets $I\subset[m]$, $J\subset[\bar n]$ of the same size, define $\Delta_{I,J}(f_G)$ to be the minor of $f_G$ with row set $I$ and column set $J$. Thus $\Delta_{I,J}(f_G)$ is a nonzero polynomial if and only if $(I,J)\in\LRS_G$. Let $N_G$ be the \emph{Newton polytope} of the product of all non-zero minors of $f_G$:
\begin{equation}\label{eq:Newton}
N_G:=\Newton \left(\prod_{(I,J)\in\LRS_G} \Delta_{I,J}(f_G)\right) \subset \R^{G}.
\end{equation}

For the case $G=K_{m,n}$ it was shown in~\cite[Example~10.C.1.3(b)]{GKZ} that $N_G$ is the \emph{secondary polytope} of $\DD$. We generalize this statement to arbitrary $G$:

\begin{proposition}\label{prop:nabla}
For $G\subset K_{m,n}$, $N_G$ is combinatorially equivalent to the secondary polytope of $Q_G$. More precisely, these polytopes have the same normal fans.
\end{proposition}
\begin{proof}
  Let us consider an $m\times n$ matrix $h=(h_{ij})\in\R^G$ such that $h_{ij}=0$ when $(i,\bar j)\notin G$. It defines a \emph{regular subdivision} $\triang_h$ of $Q_G$ as follows. Let $e_0,e_1,\dots,e_m,e_{\bar1},\dots,e_{\bar n}$ be a basis of $\R^{m+n+1}$, and consider a polytope $Q_G(h)\subset\R^{m+n+1}$ defined as the convex hull of $e_i+e_{\bar j}+h_{ij}e_0$ for all $(i,\bar j)\in G$. Then $\triang_h$ is the subdivision of $Q_G$ obtained by projecting the lower faces of $Q_G(h)$ from $\R^{m+n+1}$ to $\R^{m+n}$. It is easy to see that $\Delta_F$ is contained in a face of $\triang_h$ for some forest $F\subset G$ if and only if for each partial matching $M\subset F$ and any other matching $M'\subset G$ with $\LeftSupport(M)=\LeftSupport(M')$, $\RightSupport(M)=\RightSupport(M')$, we have
  \begin{equation}\label{eq:regular}
 \sum_{(i,\bar j)\in M} h_{ij}\leq \sum_{(i,\bar j)\in M'}h_{ij}.
  \end{equation}
  If we fix $I:=\LeftSupport(M)$ and $J:=\RightSupport(M)$ then the set of all inequalities of the form~\eqref{eq:regular} describes exactly the normal fan of $\Newton(\Delta_{I,J}(f_G))$. Thus the normal fan of the secondary polytope of $Q_G$ is their common refinement. On the other hand, $N_G$ is the Minkowski sum of $\Newton(\Delta_{I,J}(f_G))$ over all $(I,J)\in\LRS_G$, and thus its normal fan is the common refinement of the normal fans of the summands.
\end{proof}

  The vertices of the secondary polytope of $Q_G$ correspond to \emph{coherent} (or \emph{regular}) triangulations of $Q_G$. Proposition~\ref{prop:nabla} implies that each such triangulation corresponds to a vertex of $N_G$. Note that $N_G$ is a Minkowski sum of Newton polytopes of $\Delta_{I,J}(f_G)$, thus a vertex of $N_G$ corresponds to choosing a vertex inside each summand, that is, for each pair $(I,J)\in\LRS_G$, we choose a single matching $M\subset G$ with $\LeftSupport(M)=I$ and $\RightSupport(M)=J$. Note that by Lemma~\ref{lemma:LDp_bij}, part~\eqref{item:bij:LDpRDp}, any triangulation of $Q_G$ (not necessarily a coherent one) corresponds to a choice of a single matching $M$ for each pair $(I,J)\in\LRS_G$.  In the case when $G=K_{m,n}$ is the complete bipartite graph, it was shown in~\cite{OYmatching} that any such collection of matchings satisfying a certain set of axioms (most notably, the \emph{linkage axiom} of~\cite{SZ,BZ}) equals $\Matchings(\triang)$ for some triangulation $\triang$ of $\DD$. In~\cite{SZ,BZ}, the authors considered a closely related object, namely the Newton polytope of the product of \emph{maximal} minors (as opposed to \emph{all} minors as we did in~\eqref{eq:Newton}) of $f_{K_{m,n}}$. It would be interesting to generalize the constructions of~\cite{OYmatching,SZ,BZ} to arbitrary subgraphs $G\subset K_{m,n}$.


\begin{figure}
  \donthidefigure{
\begin{tabular}{c|c}
\begin{tikzpicture}[scale=0.6]
\draw[line width=1.5] (0,0) -- (30:5);
\draw[line width=1.5] (0,0) -- (150:5);
\draw[line width=1.5] (0,0) -- (-90:5);
\node[circle,draw,fill=white,text=black,scale=0.6] (p) at (0,0) {$p$};
\end{tikzpicture}&
                   \begin{tikzpicture}[scale=0.6]
\draw[line width=2,dashed,opacity=0.5] (0,0) circle (3);
\draw[line width=1.5] (0,0.2) -- (45:1)--(20:2)-- (30:3) -- (30:5);
\draw[line width=1.5] (0,0.2) -- (145:1)--(155:1.5)--(140:2)--(150:3) -- (150:5);
\draw[line width=1.5] (0,0.2) -- (-70:1)--(-90:3) -- (-90:5);
\node[circle,draw,fill=white,text=black,scale=0.6] (p) at (0,0.2) {$p$};
\node[opacity=0.5,anchor=160] (B) at (-20:3) {$B$};
\end{tikzpicture}
  \end{tabular}
  }
  \caption{\label{fig:trop_lines}A tropical line (left) and a tropical pseudoline (right).}
\end{figure}
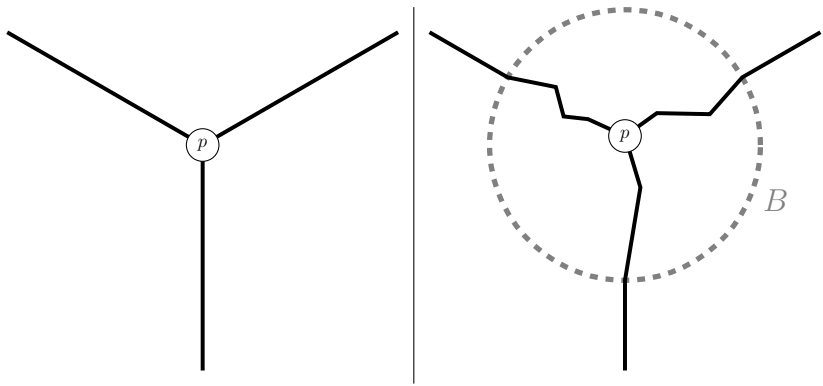

\subsection{Lozenge tilings and tropical pseudoline arrangements}\label{sect:m=3}
Throughout this section, we assume that $m=3$. We refer the reader to Figure~\ref{fig:dtphs} for some of the bijections that we mention below.

\begin{figure}
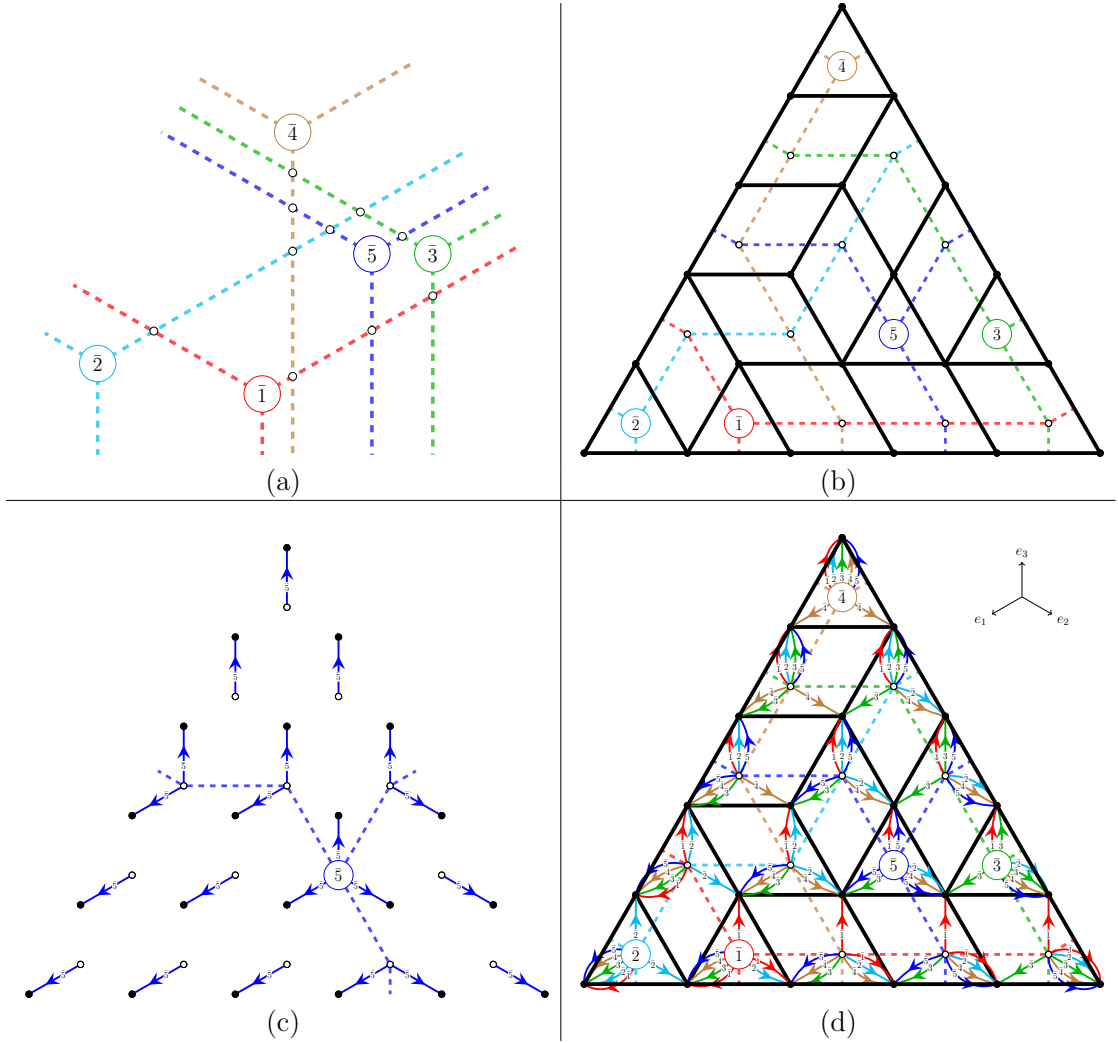


  \hidefigure{
    \scalebox{0.9}{
%
}
}
\caption{\label{fig:dtphs} The case $m=3$, $n=5$. (a) A tropical pseudoline arrangement. (b) The corresponding lozenge tiling of $n\Delta_{[m]}$. (c) The pseudoline $\bar 5$ together with the arrows $b\to a$ of $\SVTE$ for which $\bar5\in\svt(b\to a)$. (d) A tropical pseudoline arrangement, a lozenge tiling, and a trianguloid, all  corresponding to each other.}
\end{figure}

\def\tpsl{L}
Our first goal is to recast the notion of a \emph{tropical pseudohyperplane} for $m=3$ in elementary terms. See~\cite{DS,AD} for precise definitions for general $m$.

Suppose we are given three unit vectors $u_1$, $u_2$, and $u_3$ in $\R^2$ with $u_1+u_2+u_3=0$, and let $B\subset \R^2$ be the unit ball centered at the origin. Given a point $p\in B$, a \emph{tropical line} $\tpsl$ centered at $p$ is a union of three rays $r_1,r_2,r_3:\R_{\geq0}\to \R^2$ such that for each $i=1,2,3$, we have $r_i(t)=p-tu_i$ for all $t\geq0$. A \emph{tropical pseudoline} $\tpsl$ is an image of a tropical line under a piecewise-linear homeomorphism $\phi$ of $\R^2$ that fixes $\R^2\setminus B$. The image $\phi(p)$ is called the \emph{center} of $\tpsl$ and the piecewise-linear curves $\phi\circ r_i$ are called the \emph{legs} of $\tpsl$. See Figure~\ref{fig:trop_lines}.


We say that a family $\tpsl^\parr{\bar 1},\tpsl^\parr{\bar 2},\dots,\tpsl^\parr{\bar n}$ of tropical pseudolines form an \emph{arrangement} if any two of them intersect exactly once (and this intersection is transversal), the center of $\tpsl^\parr {\bar j}$ is not contained in $\tpsl^\parr {\bar k}$ for $\bar j\neq \bar k$, and no three of them intersect simultaneously at a single point. An arrangement of $5$ tropical pseudolines is shown in Figure~\ref{fig:dtphs}~(a). In this case, all of them are actual tropical lines.

\begin{remark}
We note that any arrangement of tropical lines yields a (very degenerate) \emph{honeycomb} in the sense of Knutson-Tao~\cite{KT}. Such special honeycombs provide a simple proof of the \emph{weak PRV conjecture}~\cite{PRV}, which was proven in full generality in~\cite{Kumar, Mathieu, Polo}. We refer the reader to~\cite[Section~4]{KT} for details.
\end{remark}

Since each pair of tropical pseudolines in an  arrangement must intersect exactly once, there are $n\choose 2$ points of intersection between them. Together with the $n$ centers, these $n+1\choose 2$ points can be uniquely mapped to the points in $\DZ(3,n-1)$ so that whenever two of them belong to a leg $r^\parr{\bar j}_i(t)$ of $\tpsl^\parr {\bar j}$, the one that is closer to the center of $\tpsl^\parr{\bar j}$ maps to a point in $\DZ(3,n-1)$ with a larger $i$-th coordinate. See Figure~\ref{fig:dtphs}~(b).

In fact, this gives a simple bijection between arrangements of $n$ tropical pseudolines and lozenge tilings of a \emph{holey triangle}. In the above setting, let $T_n$ be the convex hull of $nu_1,nu_2,nu_3$. A \emph{lozenge tiling} of $T_n$ is a subdivision of $T_n$ into $n\choose 2$ lozenges and $n$ \emph{upright triangles}. Here an upright triangle is the convex hull of $u_1,u_2,u_3$ (possibly shifted by some vector) and a \emph{lozenge} is a union of an upright triangle and its reflection about one of its sides. The lozenge tiling of $T_5$ corresponding to the above arrangement of $5$ tropical pseudolines is shown in solid black lines in Figure~\ref{fig:dtphs}~(b).

\begin{figure}

  \hidefigure{
\begin{tabular}{ccc}
\scalebox{1.0}{
\begin{tikzpicture}[baseline=(ZERO.base)]
\coordinate (ZERO) at (0:0.00);
\tikzset{myptr/.style={decoration={markings,mark=at position 0.7 with %
    {\arrow[scale=1.5,>=stealth]{>}}},postaction={decorate}}}
\coordinate (L3) at (90:1.50);
\coordinate (L2) at (330:1.50);
\coordinate (L1) at (210:1.50);
\coordinate (LL3) at (90:3.00);
\coordinate (LL2) at (330:3.00);
\coordinate (LL1) at (210:3.00);
\node[scale=0.3,draw,circle] (N0x0x1) at (barycentric cs:L1=0.0,L2=0.0,L3=0.5) { };
\node[scale=0.3,draw,circle] (N0x1x0) at (barycentric cs:L1=0.0,L2=0.5,L3=0.0) { };
\node[scale=0.3,draw,circle] (N1x0x0) at (barycentric cs:L1=0.5,L2=0.0,L3=0.0) { };
\node[scale=0.3,draw,circle,fill=black] (N0x1x1) at (barycentric cs:LL1=0.0,LL2=0.3333333333333333,LL3=0.3333333333333333) { };
\node[scale=0.3,draw,circle,fill=black] (N1x0x1) at (barycentric cs:LL1=0.3333333333333333,LL2=0.0,LL3=0.3333333333333333) { };
\node[scale=0.3,draw,circle,fill=black] (N1x1x0) at (barycentric cs:LL1=0.3333333333333333,LL2=0.3333333333333333,LL3=0.0) { };
\draw[draw=brown,myptr,line width=1pt] (N0x0x1) to[bend right=-30] (N0x1x1);
\draw[draw=blue,myptr,line width=1pt] (N0x0x1) to[bend right=0] (N0x1x1);
\draw[draw=pink,myptr,line width=1pt] (N0x0x1) to[bend right=30] (N0x1x1);
\draw[draw=none] (N0x0x1) to[] node[anchor=center, pos=0.3,scale=1.2] {\contour{white}{$B_1$}} (N0x1x1);
\draw[draw=cyan!80!white,myptr,line width=1pt] (N0x1x0) to[bend right=-15] (N0x1x1);
\draw[draw=green!70!black,myptr,line width=1pt] (N0x1x0) to[bend right=15] (N0x1x1);
\draw[draw=none] (N0x1x0) to[] node[anchor=center, pos=0.3,scale=1.2] {\contour{white}{$C_2$}} (N0x1x1);
\draw[draw=green!70!black,myptr,line width=1pt] (N0x0x1) to[bend right=0] (N1x0x1);
\draw[draw=none] (N0x0x1) to[] node[anchor=center, pos=0.3,scale=1.2] {\contour{white}{$A_2$}} (N1x0x1);
\draw[draw=cyan!80!white,myptr,line width=1pt] (N1x0x0) to[bend right=-15] (N1x0x1);
\draw[draw=blue,myptr,line width=1pt] (N1x0x0) to[bend right=15] (N1x0x1);
\draw[draw=none] (N1x0x0) to[] node[anchor=center, pos=0.3,scale=1.2] {\contour{white}{$C_1$}} (N1x0x1);
\draw[draw=green!70!black,myptr,line width=1pt] (N0x1x0) to[bend right=0] (N1x1x0);
\draw[draw=none] (N0x1x0) to[] node[anchor=center, pos=0.3,scale=1.2] {\contour{white}{$A_1$}} (N1x1x0);
\draw[draw=brown,myptr,line width=1pt] (N1x0x0) to[bend right=-30] (N1x1x0);
\draw[draw=blue,myptr,line width=1pt] (N1x0x0) to[bend right=0] (N1x1x0);
\draw[draw=pink,myptr,line width=1pt] (N1x0x0) to[bend right=30] (N1x1x0);
\draw[draw=none] (N1x0x0) to[] node[anchor=center, pos=0.3,scale=1.2] {\contour{white}{$B_2$}} (N1x1x0);
\node[scale=0.3,draw,circle] (N0x0x1) at (barycentric cs:L1=0.0,L2=0.0,L3=0.5) { };
\node[scale=0.3,draw,circle] (N0x1x0) at (barycentric cs:L1=0.0,L2=0.5,L3=0.0) { };
\node[scale=0.3,draw,circle] (N1x0x0) at (barycentric cs:L1=0.5,L2=0.0,L3=0.0) { };
\node[scale=0.3,draw,circle,fill=black] (N0x1x1) at (barycentric cs:LL1=0.0,LL2=0.3333333333333333,LL3=0.3333333333333333) { };
\node[scale=0.3,draw,circle,fill=black] (N1x0x1) at (barycentric cs:LL1=0.3333333333333333,LL2=0.0,LL3=0.3333333333333333) { };
\node[scale=0.3,draw,circle,fill=black] (N1x1x0) at (barycentric cs:LL1=0.3333333333333333,LL2=0.3333333333333333,LL3=0.0) { };
\node[anchor=30,scale=0.7] (BLAH) at (N1x0x0.center) {$c+e_i$};
\node[anchor=south,scale=0.7] (BLAH) at (N0x0x1.center) {$c+e_j$};
\node[anchor=150,scale=0.7] (BLAH) at (N0x1x0.center) {$c+e_k$};
\end{tikzpicture}}
&
\scalebox{2.0}{$\rightarrow$}
&
\scalebox{1.0}{
\begin{tikzpicture}[baseline=(ZERO.base)]
\coordinate (ZERO) at (0:0.00);
\tikzset{myptr/.style={decoration={markings,mark=at position 0.7 with %
    {\arrow[scale=1.5,>=stealth]{>}}},postaction={decorate}}}
\coordinate (L3) at (90:1.50);
\coordinate (L2) at (330:1.50);
\coordinate (L1) at (210:1.50);
\coordinate (LL3) at (90:3.00);
\coordinate (LL2) at (330:3.00);
\coordinate (LL1) at (210:3.00);
\node[scale=0.3,draw,circle] (N0x0x1) at (barycentric cs:L1=0.0,L2=0.0,L3=0.5) { };
\node[scale=0.3,draw,circle] (N0x1x0) at (barycentric cs:L1=0.0,L2=0.5,L3=0.0) { };
\node[scale=0.3,draw,circle] (N1x0x0) at (barycentric cs:L1=0.5,L2=0.0,L3=0.0) { };
\node[scale=0.3,draw,circle,fill=black] (N0x1x1) at (barycentric cs:LL1=0.0,LL2=0.3333333333333333,LL3=0.3333333333333333) { };
\node[scale=0.3,draw,circle,fill=black] (N1x0x1) at (barycentric cs:LL1=0.3333333333333333,LL2=0.0,LL3=0.3333333333333333) { };
\node[scale=0.3,draw,circle,fill=black] (N1x1x0) at (barycentric cs:LL1=0.3333333333333333,LL2=0.3333333333333333,LL3=0.0) { };
\draw[draw=brown,myptr,line width=1pt] (N0x0x1) to[bend right=-30] (N0x1x1);
\draw[draw=blue,myptr,line width=1pt] (N0x0x1) to[bend right=0] (N0x1x1);
\draw[draw=pink,myptr,line width=1pt] (N0x0x1) to[bend right=30] (N0x1x1);
\draw[draw=none] (N0x0x1) to[] node[anchor=center, pos=0.3,scale=1.2] {\contour{white}{$B_1$}} (N0x1x1);
\draw[draw=cyan!80!white,myptr,line width=1pt] (N0x1x0) to[bend right=-15] (N0x1x1);
\draw[draw=green!70!black,myptr,line width=1pt] (N0x1x0) to[bend right=15] (N0x1x1);
\draw[draw=none] (N0x1x0) to[] node[anchor=center, pos=0.3,scale=1.2] {\contour{white}{$C_2$}} (N0x1x1);
\draw[draw=green!70!black,myptr,line width=1pt] (N0x0x1) to[bend right=0] (N1x0x1);
\draw[draw=none] (N0x0x1) to[] node[anchor=center, pos=0.3,scale=1.2] {\contour{white}{$A_2$}} (N1x0x1);
\draw[draw=cyan!80!white,myptr,line width=1pt] (N1x0x0) to[bend right=-15] (N1x0x1);
\draw[draw=blue,myptr,line width=1pt] (N1x0x0) to[bend right=15] (N1x0x1);
\draw[draw=none] (N1x0x0) to[] node[anchor=center, pos=0.3,scale=1.2] {\contour{white}{$C_1$}} (N1x0x1);
\draw[draw=green!70!black,myptr,line width=1pt] (N0x1x0) to[bend right=0] (N1x1x0);
\draw[draw=none] (N0x1x0) to[] node[anchor=center, pos=0.3,scale=1.2] {\contour{white}{$A_1$}} (N1x1x0);
\draw[draw=brown,myptr,line width=1pt] (N1x0x0) to[bend right=-30] (N1x1x0);
\draw[draw=blue,myptr,line width=1pt] (N1x0x0) to[bend right=0] (N1x1x0);
\draw[draw=pink,myptr,line width=1pt] (N1x0x0) to[bend right=30] (N1x1x0);
\draw[draw=none] (N1x0x0) to[] node[anchor=center, pos=0.3,scale=1.2] {\contour{white}{$B_2$}} (N1x1x0);
\node[scale=0.3,draw,circle] (N0x0x1) at (barycentric cs:L1=0.0,L2=0.0,L3=0.5) { };
\node[scale=0.3,draw,circle] (N0x1x0) at (barycentric cs:L1=0.0,L2=0.5,L3=0.0) { };
\node[scale=0.3,draw,circle] (N1x0x0) at (barycentric cs:L1=0.5,L2=0.0,L3=0.0) { };
\node[scale=0.3,draw,circle,fill=black] (N0x1x1) at (barycentric cs:LL1=0.0,LL2=0.3333333333333333,LL3=0.3333333333333333) { };
\node[scale=0.3,draw,circle,fill=black] (N1x0x1) at (barycentric cs:LL1=0.3333333333333333,LL2=0.0,LL3=0.3333333333333333) { };
\node[scale=0.3,draw,circle,fill=black] (N1x1x0) at (barycentric cs:LL1=0.3333333333333333,LL2=0.3333333333333333,LL3=0.0) { };
\node[anchor=30,scale=0.7] (BLAH) at (N1x0x0.center) {$c+e_i$};
\node[anchor=south,scale=0.7] (BLAH) at (N0x0x1.center) {$c+e_j$};
\node[anchor=150,scale=0.7] (BLAH) at (N0x1x0.center) {$c+e_k$};
\draw[color=black, line width=2.0, opacity=1.0] (N0x1x1) -- (N1x1x0);
\draw[color=black, line width=2.0, opacity=1.0] (N1x0x1) -- (N1x1x0);
\draw[color=blue, line width=1.5, dashed, opacity=0.7] (N1x0x0) -- (N0x0x1);
\draw[color=green!70!black, line width=1.5, dashed, opacity=0.7] (N0x1x0) -- (N0x0x1);
\end{tikzpicture}}
\\

\end{tabular}
}
\caption{\label{fig:hexagon_convert} Reconstructing a lozenge tiling and a tropical pseudoline arrangement from a trianguloid. If $C_1\neq C_2$ then the segments of the lozenge tiling connect $c+e_i+e_k$ to $c+e_i+e_j$ and to $c+e_k+e_j$. The points $c+e_i$ and $c+e_j$ are connected by a segment (dashed) of the tropical pseudoline $L^\parr{\bar x}$, where $\{\bar x\}=C_1\setminus C_2$. Similarly, the points $c+e_k$ and $c+e_j$ are connected by a segment of the tropical pseudoline $L^\parr{\bar y}$, where $\{\bar y\}=C_2\setminus C_1$.}
\end{figure}
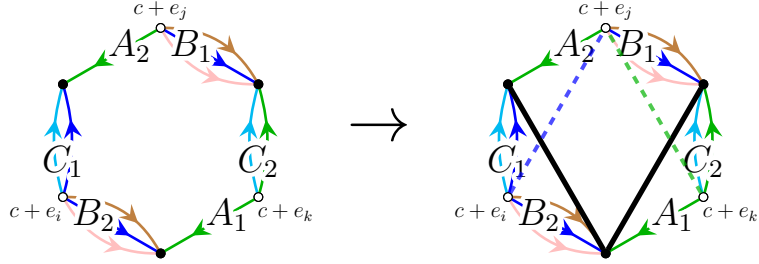

Tropical pseudoline arrangements and lozenge tilings of holey triangles are special cases for $m=3$ of \emph{tropical pseudohyperplane arrangements}~\cite{AD,Horn} and \emph{fine mixed subdivisions of $n\Delta_{[m]}$}~\cite{SantosCayley, HRS}, respectively. 

Let us now explain a way of constructing a trianguloid from these objects. Fix $\bar j\in[\bar n]$ and consider the image of $\tpsl^\parr{\bar j}$ in $(n-1)\Delta_{[3]}$, see Figure~\ref{fig:dtphs}~(c). Its complement in $(n-1)\Delta_{[3]}$ consists of three connected components. For $i=1,2,3$, we denote by $C^\parr{\bar j}_i\subset (n-1)\Delta_{[3]}$ the closure of the connected component that contains the vertex $(n-1)e_i$. Each point of $\DZ(3,n-1)$ now belongs to $C^\parr{\bar j}_i$ for one or several values of $i$. We then define a trianguloid $\svt$ by the condition that $\bar j$ belongs to $\svt(b\toidot)$ for an edge $(b\toidot)\in\SVTE$ whenever $b$ belongs to $C^\parr{\bar j}_i$. Thus for example if $b$ is the image of the center of $\tpsl^\parr{\bar j}$ then it belongs to $C^\parr{\bar j}_i$ for all $i=1,2,3$. It is easy to see that thus defined map $\svt:\SVTE\to 2^{[\bar n]}$ is indeed a trianguloid, and the corresponding collection of spanning trees of $K_{3,n}$ yields a fine mixed subdivision of $n\Delta_{[3]}$ that coincides with the lozenge tiling of $T_n$ described above.

A way of describing the inverse correspondence can be given using Axiom~\axHexagonK. Namely let $c\in\DZ(3,n-2)$ be a point and let $i,j,k$ be three indices with $\{i,j,k\}=\{1,2,3\}$ so that
\[\svt( c+e_i\tokdot)=\svt(c+e_j\tokdot )\quad \text{and}\quad \svt( c+e_j\toidot )=\svt( c+e_k\toidot ).\]
Then it is easy to see that we in fact must have $\svt(c+e_i\tojdot)\neq \svt(c+e_k\tojdot)$, i.e., the converse to Axiom~\axHexagonK holds for $m=3$. Indeed, otherwise the three sets $\svt(c+e_i\tojdot)=\svt(c+e_k\tojdot)$, $\svt( c+e_i\tokdot)=\svt(c+e_j\tokdot )$, $\svt( c+e_j\toidot )=\svt( c+e_k\toidot )$  would be pairwise disjoint (see Remark~\ref{rmk:set_partitionK}), so their union would have  size $c_i+c_j+c_k+3=n+1$ which is impossible for a subset of $[\bar n]$. Now, let us connect $c+e_i+e_k$ with $c+e_i+e_j$ and with $c+e_j+e_k$ using solid black lines, see Figure~\ref{fig:hexagon_convert}. We claim that the union of these solid black lines over all hexagons, together with the boundary of $n\Delta_{[3]}$, yields the lozenge tiling of a holey triangle corresponding to $\triang_\svt$. Similarly, denote $\bar x$ to be the unique element of $\svt(c+e_i\tojdot)\setminus \svt(c+e_k\tojdot)$ and $\bar y$ to be the unique element of $\svt(c+e_k\tojdot)\setminus \svt(c+e_i\tojdot)$. Then connect $c+e_i$ with $c+e_j$ using a dashed line labeled $\bar x$ and connect $c+e_k$ with $c+e_j$ using a dashed line labeled $\bar y$, as in Figure~\ref{fig:hexagon_convert}. We claim that the union of these dashed lines over all hexagons yields the tropical pseudoline arrangement that corresponds to $\svt$. We encourage the reader to examine the hexagons of this form in Figure~\ref{fig:dtphs}~(d), which is the superposition of a trianguloid, a tropical pseudoline arrangement, and a lozenge tiling, all corresponding to the same triangulation of $\Delta_{[3]}\times\Delta_{[\bar 5]}$.

\section{Main results: the case of arbitrary $G$}\label{sec:main}
We extend the results of Section~\ref{sec:main_K} to  arbitrary
connected subgraphs $G\subset K_{m,n}$.


Define a directed graph $\SVTG$ with vertex set $(P_G\sqcup \PGpm)\cap \Z^m$ and edge set $\SVTGE:=\{a-e_i\to a\mid a\in P_G\cap \Z^m, i\in[m]: a_i>0\}$. Note that by the definition of $\PGpm$, we have $a-e_i\in \PGpm\cap \Z^m$. We again abbreviate the edge $a-e_i\to a$ as either $a-e_i\toidot$ or $\dottoi a$.

\begin{figure}

\hidefigure{


\\

\end{tabular}
}
\caption{\label{fig:trianguloid_G} A graph $G\subset K_{m,n}$ for $m=3,n=5$ and a collection $\RM(\triang)$ of right semi-matchings for some triangulation $\triang$ of $Q_G$ (left). The corresponding trianguloid $\svt_\triang$ (right). Each black vertex $a\in P_G\cap\Z^m$ of $\svt$ corresponds to a unique right semi-matching $F_\triang(a)\in\RM(\triang)$ such that $\LDp(F)=a$. In this case, the incoming arrows of $a$ in $\svt$ in the direction of $e_i$ are labeled by the neighbors of $i$ in $F_\triang(a)$. The white vertices of $\svt$ are the lattice points of $\PGpm$.}
\end{figure}

\begin{definition}\label{dfn:pre_trianguloid}
  A \emph{pre-trianguloid} is a map $\svt: \SVTGE\to 2^{[\bar n]}$ satisfying the following axioms:
  \begin{enumerate}
  \item[(T1')]\label{ax:cardinality} for every edge $(\dottoi a) \in\SVTGE$, we have $|\svt(\dottoi a)|=a_i$.
  \item[(T2')]\label{ax:existence} for each $a\in P_G\cap \Z^m$ and $\bar j\in [\bar n]$, there exists an index $i\in N_{\bar j}(G)$ such that $\bar j\in \svt(\dottoi a )$.
  \item[(T3')]\label{ax:containment} If both $a$ and $a':=a+e_i-e_j$ belong to $P_G\cap \Z^m$ and $a_i>0$ then
    \[\svt(\dottoi a )\subset \svt(\dottoi  a').\]
  \end{enumerate}
\end{definition}

It is clear that Remark~\ref{rmk:set_partitionK} generalizes to the case of arbitrary $G$. We also note that if $a_i=0$ for some $a\in P_G\cap\Z^m$ and $i\in[m]$ then there is no edge $\dottoi a$ in $\SVTGE$ because $a-e_i\notin\PGpm$. However, in this case Axiom~\axCardinality would require $\svt(\dottoi a)$ to have zero cardinality, and in fact setting $\svt(\dottoi a):=\emptyset$ for all such pairs of $a$ and $i$ does not have any effect on our arguments.

Similarly to the case $G=K_{m,n}$, for any triangulation $\triang$ of $Q_G$ and any $b\in P_G^-\cap \Z^m$, there is a unique tree $T_\triang(b)\in\Trees(\triang)$ such that $\LDm(T_\triang(b))=b$, so we can define $\svt_\triang(b\toidot)$ by~\eqref{eq:svt_triang_K}. However, this does not define $\svt_\triang$ on all $\SVTGE$ because some edges of $\SVTG$ are of the form $b\to a$ for $b\in \PGpm\setminus P_G^-$. Instead, we use right semi-matchings from Definition~\ref{dfn:RSM}.

Recall that by Lemma~\ref{lemma:LDp_bij}, part~\eqref{item:bij:LDp}, $\LDp$ is a bijection between the set $\RM(\triang)$ of right semi-matchings  of $\triang$ and the set $P_G\cap\Z^m$ of lattice points of $P_G$. Denote by $F_\triang:P_G\cap \Z^m\to \RM(\triang)$ the inverse of this bijection. Given a triangulation $\triang$ of $Q_G$, define a map $\svt_\triang:\SVTGE\to 2^{[\bar n]}$ by
\begin{equation}\label{eq:svt_triang}
\svt_\triang(\dottoi a)=N_{i}(F_\triang(a))
\end{equation}
 for all $a\in P_G\cap \Z^m$ and all $i\in [m]$ such that $a_i>0$.

 We have the analog of Proposition~\ref{prop:triang_implies_svt_K}.
\begin{proposition}\label{prop:triang_implies_svt}
  If $\triang$ is a triangulation of $Q_G$ then $\svt_\triang$ is a pre-trianguloid.
\end{proposition}

It may seem that the definition~\eqref{eq:svt_triang} of $\svt_\triang$ for the case of arbitrary $G$ differs from the corresponding definition~\eqref{eq:svt_triang_K} for the case of $G=K_{m,n}$. The next lemma shows that this is not the case.
\begin{lemma}\label{lemma:two_definitions_agree}
  Let $\triang$ be a triangulation of $Q_G$ and define $\svt:=\svt_\triang$ by~\eqref{eq:svt_triang}. For $b\in P_G^-$, the collection $T_\svt(b)$ of edges given by~\eqref{eq:svt_triang_K} is the unique spanning tree of $G$ satisfying $\LDm(T_\svt(b))=b$ and $\Delta_{T_\svt(b)}\in\triang$.
\end{lemma}

Thus in the case $G=K_{m,n}$, the two definitions~\eqref{eq:svt_triang_K} and~\eqref{eq:svt_triang} of $\svt_\triang$ agree with each other.

\begin{proposition}\label{prop:svt_implies_trees}
For a pre-trianguloid $\svt$ and a point $b\in P_G^-\cap \Z^m$, $T_\svt(b)$ is a spanning tree of $G$.
\end{proposition}

To generalize the definition of a trianguloid to the case of an arbitrary $G$, we slightly modify Axiom~\axHexagonK for points on the boundary of $P_G^-$.

\begin{definition}
  A \emph{trianguloid} is a pre-trianguloid $\svt: \SVTGE\to 2^{[\bar n]}$ satisfying the following \emph{Hexagon axiom}:
  \begin{enumerate}
\item[(T4')] \label{ax:hexagon} let  $c\in \Z^m$ and consider three distinct indices $i,j,k\in [m]$ such that $c+e_i,c+e_k\in P_G^-$ and $\svt(c+e_i\tojdot)\neq \svt(c+e_k\tojdot )$. Then we have $c+e_j\in P_G^-$ and
 \[\svt( c+e_i\tokdot)=\svt(c+e_j\tokdot )\quad \text{and}\quad \svt( c+e_j\toidot )=\svt( c+e_k\toidot ).\]
\end{enumerate}
\end{definition}

We are ready to state our main result:

\begin{theorem}\label{thm:svt_implies_triang}
  The map $\triang\mapsto\svt_\triang$ is a bijection between triangulations of $Q_G$ and trianguloids.
\end{theorem}

Theorem~\ref{thm:bijections_different_K} also generalizes to the case of arbitrary $G$. 

\begin{theorem}\label{thm:bijections_different}
  For two different triangulations $\triang,\triang'$ of $Q_G$, the maps $\phi_{\triang}, \phi_{\triang'}$ are different as well.
\end{theorem}

\section{From triangulations to trianguloids}\label{sec:triang_implies_svt}
In this section, we show that for any triangulation $\triang$ of $Q_G$, the map $\svt_\triang:\SVTGE\to 2^{[\bar n]}$ given by~\eqref{eq:svt_triang} is a trianguloid. We work in the generality of arbitrary connected $G\subset K_{m,n}$. Before we proceed, we need to show that the map $F_\triang$ used in~\eqref{eq:svt_triang} is well defined, thus we begin by showing Lemma~\ref{lemma:LDp_bij}.


First, we discuss the compatibility condition of~\cite{Postnikov}. Given two forests $F,F'\subset G$, let $U(F,F')$ be a \emph{directed} graph with edge set
\[\{i\to \bar j\mid (i,\bar j)\in F\}\cup \{\bar j\to i\mid (i,\bar j)\in F'\}.\]
The following result generalizes Lemma~\ref{lemma:compatible:intro} and completes the proof of Proposition~\ref{prop:determined}.
\begin{lemma}\label{lemma:compatible}
  Given two forests $F,F'\subset G$, the following conditions are equivalent:
\begin{enumerate}[\normalfont(1)]
\item\label{compat:1}   the simplices $\Delta_F$, $\Delta_{F'}$ intersect by their common face;
\item\label{compat:2} the forests $F$ and $F'$ are compatible in the sense of Definition~\ref{dfn:compatible:intro};
\item\label{compat:4}  $U(F,F')$ contains no directed cycles of length $3$ or more.
\end{enumerate}
\end{lemma}
\begin{proof}
  The equivalence of~\eqref{compat:1} and~\eqref{compat:4} is proven in~\cite[Lemma~12.6]{Postnikov} for the case when $F$ and $F'$ are spanning trees of $G$, but the proof translates verbatim to the case of forests.  The fact that~\eqref{compat:4} implies~\eqref{compat:2} is obvious.
  Finally, note that if there exist $M\subset F$ and $M'\subset F'$ as in Definition~\ref{dfn:compatible:intro}, i.e., such that $\LeftSupport(M)=\LeftSupport(M')$ and $\RightSupport(M)=\RightSupport(M')$, then for every vertex of $G$, its indegree in $U(M,M')$ equals to its outdegree in $U(M,M')$, and thus $U(M,M')$ contains a directed cycle of length at least $4$ because we have assumed $M\neq M'$. This shows that~\eqref{compat:2} implies~\eqref{compat:4}, finishing the proof of the lemma.
\end{proof}

We need one more step before proving Lemma~\ref{lemma:LDp_bij}.

\begin{lemma}\label{lemma:LDp_different}
Suppose that $\triang$ is a triangulation of $Q_G$ and consider two forests $F,F'\in\RM(\triang)$. Then we have $\LDp(F)\neq \LDp(F')$ for $F\neq F'$.
\end{lemma}
\begin{proof}
Since $F$ and $F'$ both belong to $\triang$, they must be compatible by Lemma~\ref{lemma:compatible}, but on the other hand, the assumptions of Lemma~\ref{lemma:LDp_different} force $F$ and $F'$ to satisfy $\deg_i(F)=\deg_i(F')$ and $\deg_{\bar j}(F)=\deg_{\bar j}(F')$ for all $i\in [m]$ and $\bar j\in[\bar n]$. Thus there is a directed cycle in $U(F,F')$, hence they are not compatible, a contradiction.
\end{proof}

\begin{proof}[Proof of Lemma~\ref{lemma:LDp_bij}]
Parts~\eqref{item:bij:LD} and~\eqref{item:bij:RD} follow from Lemma~\ref{lemma:LD_bij}.

We now prove part~\eqref{item:bij:LDp}. Let $\triang$ be a triangulation of $Q_G$. By Lemma~\ref{lemma:LDp_different}, we only need to show that for every lattice point $a\in P_G\cap \Z^m$, there exists a forest $F\in\RM(\triang)$ such that $\LDp(F)=a$. It is shown in~\cite[Section~14]{Postnikov} that $\triang$ corresponds to a \emph{fine mixed subdivision} of $P_G$ and it follows from the proof of~\cite[Proposition~14.12]{Postnikov} that $a$ is a vertex of that mixed subdivision. It remains to note that such a vertex corresponds precisely to a simplex $\Delta_F$ for some $F\in\RM(\triang)$. We are done with the proof of part~\eqref{item:bij:LDp}. Part~\eqref{item:bij:RDp} is completely analogous.

Finally, we show part~\eqref{item:bij:LDpRDp}. Let $(I,J)\in\LRS_G$. By Lemma~\ref{lemma:compatible}, there is at most one matching $F\in\Matchings(\triang)$ such that $\LeftSupport(F)=I$ and $\RightSupport(F)=J$, and it remains to show that such a matching exists. Indeed, denote $k:=|I|=|J|$ and consider the point
\[p_{I,J}:=\frac1k \left(e_I+e_J\right).\]
Since $(I,J)\in\LRS_G$, we get that $p_{I,J}\in Q_G$, and therefore it must belong to $\Delta_T$ for some $T\in \Trees(\triang)$, in other words, there is a way to represent $p_{I,J}$ as a convex combination of vectors $e_i+e_{\bar j}$ for $(i,\bar j)\in T$. Let $F\subset T$ be the set of edges whose coefficients in this convex combination are nonzero. We claim that $F$ is a partial matching with $\LeftSupport(F)=I$ and $\RightSupport(F)=J$. Indeed, let $i\in [m]$ be a leaf of $F$ adjacent to a single edge $(i,\bar j)\in F$. It follows that $i\in I$, $\bar j\in J$, and the coefficient of $(i,\bar j)$ in the convex combination must be equal to $\frac1k$. Therefore $\bar j$ is not adjacent to any other edges of $F$. Since this holds for every leaf $i$ of $F$ (and similarly for every leaf $\bar j$ of $F$), we have shown that $F$ is a partial matching, thus finishing the proof of part~\eqref{item:bij:LDpRDp}.
\end{proof}

Lemma~\ref{lemma:two_definitions_agree} follows from our next observation.

\begin{lemma}\label{lemma:tree_forest}
Let $\triang$ be a triangulation of $Q_G$ and $i\in [m]$. Consider a tree $T\in\Trees(\triang)$ and a forest $F\in\RM(\triang)$ satisfying $\LDm(T)+e_i=\LDp(F)$. Then $N_i(F)=N_i(T)$.
\end{lemma}
\begin{proof}
  For each $i'\in [m]$ that is not equal to $i$, we have $\deg_{i'}(T)=\deg_{i'}(F)+1$, so there exists a map $\bar q:[m]\setminus \{i\}\to [\bar n]$ satisfying $(i',\bar q(i'))\in T\setminus F$ for all $i'\in [m]\setminus \{i\}$. Suppose that $N_i(F)\neq N_i(T)$. Then the map $\bar q$ can be extended to $[m]$ by setting $\bar q(i)$ to be any element of $N_i(T)\setminus N_i(F)$ (these two sets are of the same cardinality). After that, we have $(i',\bar q(i'))\in T\setminus F$ for all $i'\in [m]$. Thus the directed graph $U(T,F)$ contains a directed subgraph $U'$ with edge set
  \begin{equation}\label{eq:q_bar}
  \{i'\to \bar q(i')\mid i'\in[m]\}\cup \{\bar j\to i'\mid (i',\bar j)\in F\}.
  \end{equation}
  By construction, $U'$ has no directed cycles of length $2$, and each vertex of this directed graph has outdegree $1$. Thus $U'$ contains a directed cycle, a contradiction.
\end{proof}

We now fix a triangulation $\triang$ of $Q_G$ and proceed to showing that the map $\svt_\triang$ satisfies the axioms of a trianguloid.

\begin{lemma}\label{lemma:axCardinality}
  The map $\svt_\triang$ satisfies Axioms~\axCardinality and~\axExistence.
\end{lemma}
\begin{proof}
  This is obvious from~\eqref{eq:svt_triang}: $\svt_\triang(\dottoi a)$ is equal to $N_i(F_\triang(a))$, so its cardinality is equal to the degree of $i$ in $F_\triang(a)$, i.e., to $a_i$ (by the definition of $F_\triang$), which proves~\axCardinality. Since $F_\triang(a)$ is a right semi-matching, for each $\bar j\in [\bar n]$ there exists a (unique) $i\in [m]$ such that $(i,\bar j)\in F_\triang(a)$, which proves~\axExistence.
\end{proof}

\begin{lemma}\label{lemma:axContainment}
  The map $\svt_\triang$ satisfies Axiom~\axContainment.
\end{lemma}
\begin{proof}
Let $a$ and $a':=a+e_i-e_j$ be two points of $P_G\cap \Z^m$, and let $F:=F_\triang(a)$, $F':=F_\triang(a')$ be the corresponding elements of $\RM(\triang)$. Consider the directed subgraph $U'$ of $U(F,F')$ with all edges of $F\cap F'$ removed. We would like to show that $N_i(F)\subset N_i(F')$. Suppose that this is not the case, then clearly whenever a vertex of $U'$ has an incoming edge, it also must have an outgoing edge. (This was already true for each vertex of $U'$ except for possibly $i$.) Thus $U'$ contains a directed cycle and we get a contradiction.
\end{proof}

We have thus shown that $\svt_\triang$ is a pre-trianguloid, completing the proof of Proposition~\ref{prop:triang_implies_svt} as well as of its special case, Proposition~\ref{prop:triang_implies_svt_K}. We finish by showing that $\svt_\triang$ is in fact a trianguloid.

\begin{lemma}\label{lemma:axHexagon}
  The map $\svt:=\svt_\triang$ satisfies Axiom~\axHexagon.
\end{lemma}
\begin{proof}
  Let $c$, $i$, $j$, $k$ be as in Axiom~\axHexagon, and let $T:=T_\triang(c+e_i)\in\Trees(\triang)$ be the tree with $\LDm(T)=c+e_i$. Let $F:=F_\triang(c+e_j+e_k)\in\RM(\triang)$ be the forest with $\LDp(F)=c+e_j+e_k$. Assume that $\svt(c+e_i\tojdot)\neq\svt(c+e_k\tojdot)$, which by Lemma~\ref{lemma:tree_forest} is equivalent to $N_j(T)\neq N_j(F)$, since $N_j(F)=\svt(c+e_k\tojdot)$ and $N_j(T)=\svt(c+e_i\tojdot)$. These two sets have the same cardinality $c_j+1$, so there exists $\bar q(j)\in [\bar n]$ such that $(j,\bar q(j))\in T\setminus F$. For each $i'\in [m]\setminus\{j,k\}$, we have $\deg_{i'}(T)>\deg_{i'}(F)$ and thus there exists $\bar q(i')\in [\bar n]$ such that $(i',\bar q(i'))\in T\setminus F$. Hence if $N_k(T)\neq N_k(F)$ then we can extend $\bar q$ to the whole $[m]$ and get a contradiction because $U(T,F)$ will contain a directed subgraph $U'$ with edge set given by~\eqref{eq:q_bar} that must have a cycle of length more than $2$. We have shown that $N_k(T)=N_k(F)$, equivalently, $\svt(c+e_i\tokdot)=\svt(c+e_j\tokdot)$. The proof that $\svt(c+e_j\toidot)=\svt(c+e_k\toidot)$ is completely similar.

  The only thing left to show is that $c+e_j$ belongs to $P_G^-$. It suffices to show that for any $t\in [m]$ we have $c+e_j+e_t\in P_G$. This is clear for $t=k$ so assume that $t\neq k$. We claim that there exists a sequence $(t_1,t_2,\dots,t_r)$ of distinct elements of $[m]$ such that $t_1=t$, $t_r=k$, and for each $1\leq s\leq r-1$, $t_{s+1}$ is the unique vertex that is connected to $\bar q(t_s)$ in $F$. Indeed, we build such a sequence by induction: if $t_s\neq k$ then there exists a unique $t_{s+1}$ connected to $\bar q(t_s)$ in $F$, and it must be different from $t_1,\dots,t_s$ since otherwise we would have found a directed cycle in $U(T,F)$. This process has to terminate, and since $k$ is the only vertex for which $\bar q$ is undefined, we must have $t_r=k$ for some $r\geq 2$. Consider now the forest
  \[F'=F\cup\{(t_s,\bar q(t_s))\mid 1\leq s\leq r-1\}\setminus \{(t_{s+1},\bar q(t_s))\mid 1\leq s\leq r-1\}.\]
  Clearly $F'$ is a right semi-matching with $\LDp(F')=c+e_j+e_t$ which implies that $c+e_j+e_t\in P_G$. We are done with the proof.
\end{proof}

\newcommand\todot[1]{\too{#1}\dotDown}
\newcommand\dotto[1]{\dotUp\too{#1}}

\section{From trianguloids to triangulations}\label{sec:svt_implies_triang}

We start by showing Propositions~\ref{prop:svt_implies_trees} and~\ref{prop:svt_implies_trees_K}. The following lemma will be used many times throughout our proofs.

\begin{lemma}\label{lemma:path}
  Let $\svt$ be a pre-trianguloid, $b\in P_G^-$, and let $T:=T_\svt(b)\subset G$ be the subgraph given by~\eqref{eq:svt_triang_K}. Suppose that for some $r\geq 1$ there exists a \emph{simple path} $\bar j_1,i_1,\bar j_2,\dots,i_{r-1}, \bar j_r,i_r$ in $T$, i.e., there exist distinct indices $i_1,\dots,i_r\in [m]$ and $\bar j_1,\dots,\bar j_r\in [\bar n]$ such that for each $1\leq s\leq r$ we have $(i_{s},\bar j_{s})\in T$ and for each $1\leq s\leq r-1$ we have $(i_s,\bar j_{s+1})\in T$. Then for all $1\leq s\leq t\leq r$, we have
  \begin{equation}\label{eq:path}
\bar j_s\in \svt(\dotto{i_s} b+e_{i_t}).
  \end{equation}
\end{lemma}
\begin{proof}
  
We note that for each $1\leq t\leq r$, we have $b+e_{i_t}\in P_G$, so the edge $\dotto{i_s} b+e_{i_t}$ belongs to $\SVTGE$ for all $1\leq s\leq t\leq r$. This includes the statement that the vector $b+e_{i_t}-e_{i_s}$ has nonnegative coordinates for all $1\leq s\leq t\leq r$.

  Fix $t\geq1$. For $s=t$, the statement $\bar j_t\in\svt(\dotto{i_t}b+e_{i_t})$ is by definition equivalent to $ (i_t,\bar j_t)\in T$, which is the case for all $1\leq t\leq r$. Let now $1\leq s<t$ and suppose that~\eqref{eq:path} is proven for the pair $(s+1,t)$. By Axiom~\axContainment, $\svt(\dotto{i_s} b+e_{i_t})\subset \svt(\dotto{i_s} b+e_{i_s})$, and by Axiom~\axCardinality, their cardinalities satisfy
  \[|\svt(\dotto{i_s} b+e_{i_t})|+1=|\svt(\dotto{i_s} b+e_{i_s})|.\]
  Therefore there exists a unique index $\bar j\in \svt(\dotto{i_s} b+e_{i_s})\setminus \svt(\dotto{i_s} b+e_{i_t})$. We claim that $\bar j=\bar j_{s+1}$. Indeed, we know that $\bar j_{s+1}\in \svt(\dotto{i_s} b+e_{i_s})$ because $(i_s,\bar j_{s+1})\in T$. On the other hand, by the induction hypothesis for $(s+1,t)$, we get that $\bar j_{s+1}\in\svt(\dotto{i_{s+1}} b+e_{i_t})$. Therefore by Remark~\ref{rmk:set_partitionK}, $\bar j_{s+1}\notin\svt(\dotto{i_{s}} b+e_{i_t})$ and we have shown that $\bar j=\bar j_{s+1}$. Since $\bar j_s\neq \bar j_{s+1}$ and $\bar j_s\in \svt(\dotto{i_s} b+e_{i_s})$, it follows that $\bar j_s\in\svt(\dotto{i_s} b+e_{i_t})$, which completes the proof. 
\end{proof}

\begin{remark}
We will sometimes use a variant of Lemma~\ref{lemma:path} where the path starts with $i_1$ instead of $\bar j_1$ (but still ends with $i_r$). In this case, \eqref{eq:path} applies to all $2\leq s\leq t\leq r$, with the same proof.
\end{remark}

\begin{proof}[Proof of Propositions~\ref{prop:svt_implies_trees} and~\ref{prop:svt_implies_trees_K}.]
Consider a vertex $b\in P_G^-$. We need to show that the collection $T:=T_\svt(b)$ of edges gives a spanning tree of $G$. By Axiom~\axCardinality, $T$ contains exactly $\sum_{i\in[m]} (b_i+1)=n+m-1$ edges. Thus we only need to show that it contains no cycles.

Suppose that $T$ contains a cycle that consists of vertices $\bar j_1,i_1,\bar j_2, i_2 \dots,\bar j_r,i_r,\bar j_1$ in this order. By Lemma~\ref{lemma:path} applied to $s=1$ and $t=r$, we get that $\bar j_1\in \svt(\dotto{i_1} b+e_{i_r})$. On the other hand, $(i_r,\bar j_1)\in T$ so $\bar j_1\in\svt(\dotto{i_r} b+e_{i_r})$. This contradicts Remark~\ref{rmk:set_partitionK}.
\end{proof}

Let us also use Lemma~\ref{lemma:path} to prove the following result which will be used in Section~\ref{sec:bijections_different}.

\def\Sources{S}
\begin{lemma}\label{lemma:N_bar_j}
  For a pre-trianguloid $\svt$ and an index $\bar j\in[\bar n]$, consider the set 
  \[P_G(\bar j;\svt):=\{a\in P_G\cap\Z^m\mid \exists\, b\in P_G^-:\ (b\to a)\in\SVTGE\text{ and }\bar j\in \svt(b\to a)\}.\]
  Then we have
  \begin{equation}\label{eq:N_bar_j}
P_G(\bar j;\svt)=\{b+e_i\mid b\in P_G^-,\ i\in N_{\bar j}(G)\}.
  \end{equation}
\end{lemma}
Note that the right hand side of~\eqref{eq:N_bar_j} does not depend on $\svt$.
\begin{proof}
  First of all, it is clear that the left hand side of~\eqref{eq:N_bar_j} is contained in the right hand side, since if $b$ belongs to $P_G^-$ and $\bar j\in \svt(b\to a)$ then by Axiom~\axExistence, we have $a=b+e_i$ for some $i\in N_{\bar j}(G)$. Now assume that $a\in P_G\cap \Z^m$ does not belong to the left hand side of~\eqref{eq:N_bar_j}. Let $i$ be the (unique by Remark~\ref{rmk:set_partitionK}) index such that $\bar j\in\svt(\dottoi a)$. Then our assumption implies $a-e_i\notin P_G^-$. If $a$ does not belong to the right hand side of~\eqref{eq:N_bar_j} then we are done. Otherwise let $k\in N_{\bar j}(G)$ be an index such that $a-e_k\in P_G^-$. Consider the tree $T:=T_\svt(a-e_k)$, and let $\bar j_1, i_1,\bar j_2,\dots, i_r$ be the path in $T$ that connects $\bar j=\bar j_1$ to $k=i_r$. By Lemma~\ref{lemma:path} applied to $s=1$ and $t=r$, we get $\bar j\in \svt(\dotto{i_1} a)$. By Remark~\ref{rmk:set_partitionK}, this implies that $i_1=i$. Therefore
  \[T':=\left(T\setminus \{(i,\bar j)\}\right)\cup \{(k,\bar j)\}\]
  is again a tree, and since it satisfies $\LDm(T')=a-e_i$, we get that  $a-e_i\in P_G^-$, which contradicts our assumption. We are done with the proof.
\end{proof}

Finally, we focus on proving Theorems~\ref{thm:svt_implies_triang} and~\ref{thm:svt_implies_triang_K}.

\begin{lemma}\label{lemma:are_compatible}
Let $\svt$ be a pre-trianguloid and consider two trees $T:=T_\svt(c+e_i)$, $T':=T_\svt(c+e_k)$ for some $i\neq k\in [m]$ and $c$ such that $c+e_i,c+e_k\in P_G^-$. Suppose in addition that $|T\cap T'|=m+n-2$. Then $T$ and $T'$ are compatible.
\end{lemma}
\begin{proof}
Let $F:=T\cap T'$ and suppose that $T$ and $T'$ are not compatible. This is equivalent to saying that $i$ and $k$ belong to the same connected component of $F$, so consider a path $i_1,\bar j_2,i_2,\bar j_3,\dots,\bar j_r, i_r$ in $F$ such that $i_1=i$ and $i_r=k$ (here $r\geq 2$). Applying Lemma~\ref{lemma:path} to $b=c+e_i$, $s=2$, and $t=r$ shows $\bar j_2\in \svt(\dotto{i_2} c+e_i+e_k)$. On the other hand, the edge $(i,\bar j_2)$ belongs to $T'=T_\svt(c+e_k)$ and thus $\bar j_2\in \svt(\dotto{i} c+e_i+e_k)$. This contradicts Remark~\ref{rmk:set_partitionK} since we have assumed $i\neq i_2$.
\end{proof}

\begin{proposition}\label{prop:replaceable}
  Let $T\subset G$ be a spanning tree of $G$ and consider an edge $(v,\bar u)\in T$. Define $F:=T\setminus\{(v,\bar u)\}$. Then the following conditions are equivalent:
  \begin{enumerate}[\normalfont (i)]
  \item\label{item:replaceable_geom} The simplex $\Delta_F$ is not contained inside the boundary of $Q_G$.
  \item\label{item:replaceable_comb} There exists an edge $(v',\bar u')\in G$ such that $v'$ (resp., $\bar u'$) belongs to the connected component of $F$ that contains $\bar u$ (resp., $v$). 
  \end{enumerate}
\end{proposition}

If~\eqref{item:replaceable_geom} or~\eqref{item:replaceable_comb} holds then we call $(v,\bar u)$ a  \emph{replaceable edge} of $T$.

\begin{proof}
  To show that~\eqref{item:replaceable_comb} implies~\eqref{item:replaceable_geom}, observe that the trees $T$ and $T':=F\cup\{(v',\bar u')\}$ are compatible, and thus the corresponding top-dimensional simplices $\Delta_T,\Delta_{T'}\subset Q_G$ intersect by $\Delta_F$. Therefore the relative interior of $\Delta_F$ is contained inside the relative interior of $Q_G$. Conversely, suppose that there is no edge $(v',\bar u')$ satisfying~\eqref{item:replaceable_comb}. Define $I\subset [m]$, $J\subset[\bar n]$ so that $I\cup J$ is the connected component of $F$ containing $v$. It follows that there are no edges in $G$ between $J$ and $[m]\setminus I$, and that there are no edges between $I$ and $[\bar n]\setminus J$ in $F$. Consider a linear function $h:\R^{m+n}\to\R$ defined by
  \[h(x_1,\dots,x_m,x_{\bar1},\dots,x_{\bar n}):= \sum_{i\in [m]\setminus I}x_i+\sum_{\bar j\in J} x_{\bar j}.\]
Since there are no edges in $G$ between $J$ and $[m]\setminus I$, the value of $h$ on $e_i+e_{\bar j}$ for $(i,\bar j)\in G$ is at most $1$. Thus the maximum value of $h$ on $Q_G$ is $1$. On the other hand, for every edge $(i,\bar j)\in F$ we have either $i\in [m]\setminus I$ or $\bar j\in J$. Thus $h$ is identically equal to $1$ on $\Delta_F$. Since $T$ contains the edge $(v,\bar u)$ and $h(e_v+e_{\bar u})=0$, we get that the maximum of $h$ is attained at a facet of $Q_G$ that contains $\Delta_F$, finishing the proof of the proposition.
\end{proof}

\begin{lemma}\label{lemma:replaceable}
Let $\svt$ be a trianguloid and consider a point $b\in P_G^-$. Then for any replaceable edge $(v,\bar u)$ of $T_\svt(b)$, there exists a point $b'\in P_G^-$ such that $T_\svt(b)\setminus T_\svt(b')=\{(v,\bar u)\}$.
\end{lemma}

\def\dottov{{\dotto v}}
\def\tovdot{{\todot v}}
\def\dottot{{\dotto t}}
\def\totdot{{\todot t}}

\begin{proof}
Denote $c:=b-e_v$, so $(v,\bar u)$ is a replaceable edge of $\T(c+e_v)$. Let
\[B:= \{i\in[m]:\ c+e_i\in P_G^-\},\]
and consider a subset $B'\subset B$ consisting of all indices $i\in B$ such that
\begin{equation}\label{eq:B'}
\bar u\notin \svt(c+e_i\tovdot).
\end{equation}

Note that $v\in B$ and $v\notin B'$. We claim that the set $B'$ is non-empty. Indeed, let $(v',\bar u')\in G$ be any edge satisfying the conditions of Proposition~\ref{prop:replaceable}, part~\eqref{item:replaceable_comb}.  Our goal is to prove that $v'$ belongs to $B'$. First note that $c+e_{v'}\in P_G^-$ since replacing $(v,\bar u)$ with $(v',\bar u')$ in $T$ produces a spanning tree $T'$ of $G$ with $\LDm(T')=c+e_{v'}$. We now need to show that~\eqref{eq:B'} holds for $i:=v'$. 

Consider the path $i_1,\bar j_2,i_2,\dots,\bar j_r,i_r$  from $i_1:=v$ to $i_r:=v'$ in $T$. It must pass through the edge $(v,\bar u)$, thus $\bar j_2=\bar u$. Applying Lemma~\ref{lemma:path} to $b=c+e_v$, $s=2$ and $t=r$ yields
\[\bar u\in \svt(\dotto{i_2} c+e_{v}+e_{v'}).\]
By Remark~\ref{rmk:set_partitionK}, we therefore have
\[\bar u\notin \svt(\dotto{v} c+e_v+e_{v'}).\]
Thus indeed~\eqref{eq:B'} holds for $i:=v'$, and we have shown that $B'\neq\emptyset$.


Recall that $v\notin B'$. For $i\in B'$ and $j\in [m]\setminus \{v\}$, define
\[M(i,j)=
  \begin{cases}
    1, & \text{if $j=i$ or $\svt(c+e_v\tojdot)= \svt(c+e_i\tojdot);$}\\
    0, &\text{otherwise.}\\
  \end{cases} \]

Our main goal is to find $i\in B'$ such that $M(i,j)=1$ for all $j\neq v$. Indeed, for such $i$ we clearly have $\T(c+e_v)\setminus \T(c+e_i)=\{(v,\bar u)\}$. We first show that for any $i\in B'$ and $j\notin B'$ such that $j\neq v$ we have $M(i,j)=1$. Indeed, suppose that $j\notin B'$. If $j\notin B$ then $c+e_j\notin P_G^-$ so by Axiom~\axHexagon, we must have $M(i,j)=1$. If $j\in B\setminus B'$ then we have $\bar u\in \svt(c+e_j\tovdot)$ but $\bar u\notin \svt(c+e_i\tovdot)$. Thus $\svt(c+e_j\tovdot)\neq \svt(c+e_i\tovdot).$
Applying Axiom~\axHexagon to these three indices, we immediately get $M(i,j)=1$. We have shown that $M(i,j)=1$ for all $i\in B'$ and $j\notin B'\cup\{v\}$. It remains to find $i\in B'$ such that for all $j\in B'$ we have $M(i,j)=1$.

Axiom~\axHexagon imposes certain restrictions on $M(i,j)$. First, applying it to distinct indices $v,i,j$, we get
\begin{equation}\label{eq:M2}
M(i,j)+M(j,i)>0\quad\text{for all $i,j\in B'$}.
\end{equation}
Second, applying it to distinct indices $i,j,k\in B'$ yields the following:
\begin{equation}\label{eq:M3}
\text{if $M(i,k)\neq M(j,k)$ then $M(k,i)=M(j,i)$ and $M(k,j)=M(i,j)$.}
\end{equation}
Indeed, $M(i,k)\neq M(j,k)$ implies that $\svt(c+e_i\tokdot)\neq\svt(c+e_j\tokdot)$, because one of these two sets is equal to $\svt(c+e_v\tokdot)$ while the other one is not. Applying Axiom~\axHexagon, we get that $\svt(c+e_k\toidot)=\svt(c+e_j\toidot)$ and $\svt(c+e_k\tojdot)=\svt(c+e_i\tojdot)$. These conditions imply that $M(k,i)=M(j,i)$ and $M(k,j)=M(i,j)$, respectively.

We now prove that any $r\times r$ matrix $M(i,j)$ satisfying $M(i,i)=1$ for all $i$ together with \eqref{eq:M2} and~\eqref{eq:M3}, has a row filled with ones. We do this by induction on the size $r=|B'|$ of $M$. We have shown that $B'$ is non-empty, so the base case is $r=1$ which is clear. Suppose that $r>1$ and by induction we may assume that there exists $1\leq i<r$ such that $M(i,j)=1$ for all $1\leq j<r$. If $M(i,r)=1$ then we are done, so suppose that $M(i,r)=0$. We are going to show that in this case, $M(r,j)=1$ for all $1\leq j\leq r$. First, by~\eqref{eq:M2}, $M(r,i)=1$. We also know that $M(r,r)=1$. Now consider $1\leq j<r$ such that $j\neq i$. If $M(j,r)=0$ then by~\eqref{eq:M2}, $M(r,j)=1$ and we are done. So suppose that $M(j,r)=1$. Then $M(j,r)\neq M(i,r)$, so applying~\eqref{eq:M3} yields $M(r,j)=M(i,j)$, and by the induction hypothesis, $M(i,j)=1$. We have shown that $M(r,j)=1$ for all $1\leq j\leq r$, thus finishing the induction step together with the proof of the lemma.
\end{proof}

We are now ready to prove our main result.

\def\Trees{\Tcal}
\begin{proof}[Proof of Theorem~\ref{thm:svt_implies_triang}]
  Note that we have already shown in Section~\ref{sec:triang_implies_svt} that if $\triang$ is a triangulation of $Q_G$ then $\svt_\triang$ defined by~\eqref{eq:svt_triang} is a trianguloid. Suppose now that $\svt$ is a trianguloid, and let $\Trees(\svt)=\{T_\svt(b)\mid b\in P_G^-\cap\Z^m\}$ be the corresponding collection of trees. We would like to show that the simplicial complex $\triang_\svt$ whose top-dimensional simplices are $\{\Delta_T\mid \Trees(\svt)\}$ is a triangulation of $Q_G$.

  So far we have the following situation:
  \begin{enumerate}[\quad (1)]
  \item\label{item:volume} $\{\Delta_T\mid \Trees(\svt)\}$ is a collection of top-dimensional simplices inside $Q_G$, each of them has the same volume, and their total volume equals the volume of $Q_G$;
  \item\label{item:facet} for every $T\in\Trees(\svt)$, if a facet $\Delta_F$ of $\Delta_T$ is not contained inside the boundary of $Q_G$ then there exists $T'\in\Trees(\svt)$ such that $\Delta_T\cap\Delta_{T'}=\Delta_F$.
  \end{enumerate}
  Indeed, as we have already noted, Claim~\eqref{item:volume} is explained in~\cite[Lemma~12.5]{Postnikov}. Claim~\eqref{item:facet} is proven as follows. Let $(v,\bar u)$ be the unique edge in $T\setminus F$. Note that $\Delta_F$ not being contained inside the boundary of $Q_G$ by Proposition~\ref{prop:replaceable} implies that $(v,\bar u)$ is a replaceable edge in $T$. Then by Lemma~\ref{lemma:replaceable}, there exists another tree $T'\in\Trees(\svt)$ such that $T\setminus T'=\{(v,\bar u)\}$. Finally, by Lemma~\ref{lemma:are_compatible}, the trees $T$ and $T'$ are compatible, and thus $\Delta_T\cap\Delta_{T'}=\Delta_F$.

  We need to prove two claims:
  \begin{enumerate}[\quad (a)]
  \item\label{item:union} The set $R_\svt:=\cup_{T\in\Trees(\svt)} \Delta_T$ equals $Q_G$.
  \item\label{item:intersect_properly} For $T,T'\in\Trees(\svt)$, we have $\Delta_T\cap \Delta_{T'}=\Delta_{T\cap T'}$.
  \end{enumerate}
  To show~\eqref{item:union}, choose any point $q\in Q_G$ and suppose that it does not belong to $R_\svt$. Choose a generic point $r\in R_\svt$ and find the smallest $0<t<1$ such that $p:=(1-t)q+tr\in R_\svt$. Since $r$ is generic, $p$ must belong to a facet $\Delta_F$ of $\Delta_T$ for some $T\in\Trees(\svt)$. This facet is clearly not contained inside the boundary of $Q_G$, thus there exists another tree $T'\in\Trees(\svt)$ such that $F\subset T'$ and $\Delta_T\cap\Delta_T'=\Delta_F$. We get a contradiction with the minimality of $t$, thus finishing the proof of~\eqref{item:union}. Since the total volume of the simplices in $\triang_\svt$ equals the volume of $Q_G$, it follows that any two simplices in $\triang_\svt$ have disjoint interiors.

  We now prove~\eqref{item:intersect_properly}. Let $F$ be any forest of $G$ such that $F\subset T$ for some $T\in\Trees(\svt)$ and let $\Trees_F(\svt)=\{T\in\Trees(\svt)\mid F\subset T\}$. Choose any point $f$ in the relative interior of $\Delta_F$. By the same argument as in the previous paragraph, there exists a small ball $B$ around $f$ that is fully contained inside $R_F:=\cup_{T\in\Trees_F(\svt)} \Delta_T$. (Indeed, we just choose $B$ to be such that for all $T\in\Trees_F(\svt)$ and any facet of $\Delta_{T}$ that does not contain $f$, $B$ does not intersect the hyperplane containing this facet.)

  Thus for any $T\in \Trees(\svt)\setminus \Trees_F(\svt)$, $\Delta_{T}$ cannot contain $f$ because then its interior will intersect $B$ and thus it will also intersect the interior of $\Delta_{T'}$ for some $T'\in \Trees_F(\svt)$. We have shown~\eqref{item:intersect_properly} which finishes the proof of Theorems~\ref{thm:svt_implies_triang} and~\ref{thm:svt_implies_triang_K}. Note also that the map $\svt\mapsto \triang_\svt$ is inverse to the map $\triang\mapsto \svt_\triang$ by Lemma~\ref{lemma:two_definitions_agree}.
\end{proof}

\section{Proof of Theorems~\ref{thm:bijections_different_K} and~\ref{thm:bijections_different}}\label{sec:bijections_different}
\def\RED{{A}}
\def\BLUE{{A'}}
\def\RB{{B}}
\def\BR{{B'}}
\def\RBN{{M}}
\def\BRN{{M'}}

Suppose that $\triang$ and $\triang'$ are two different triangulations of $Q_G$ and let $\svt:=\svt_\triang$, $\svt':=\svt_{\triang'}$ be the corresponding trianguloids. We are going to show that the maps $\phi_\triang,\phi_{\triang'}: P_G^-\cap\Z^m\to P_{G^\ast}^-\cap\Z^n$ must be different. 

For the sake of contradiction, assume that the maps $\phi_\triang$ and $\phi_{\triang'}$ are the same. Thus for each $b\in P_G^-\cap\Z^m$ and each $\bar u\in[\bar n]$, the degree of $\bar u$ in $T_\svt(b)$ equals the degree of $\bar u$ in $T_{\svt'}(b)$. However, since the triangulations themselves are different, there exists $b\in P_G^-\cap\Z^m$ and $\bar u\in[\bar n]$ such that $N_{\bar u}(T_\svt(b))\neq N_{\bar u}(T_{\svt'}(b))$. Let us fix this $\bar u$ and show that for some $b\in P_G^-\cap\Z^m$, the degrees of $\bar u$ in $T_{\svt}(b)$ and $T_{\svt'}(b)$ must be different. We introduce the following two subsets of $\SVTGE$:
\begin{equation*}
  \begin{split}
    \RED&:=\{(b\to a)\in\SVTGE\mid a\in P_G\cap\Z^m,\  b\in P_G^-\cap\Z^m: \bar u\in \svt(b\to a)\};\\
    \BLUE&:=\{(b\to a)\in\SVTGE\mid a\in P_G\cap\Z^m,\  b\in P_G^-\cap\Z^m: \bar u\in \svt'(b\to a)\}.
  \end{split}
\end{equation*}

\begin{lemma}\label{lemma:exactly_one}
  For every  point $a\in P_G$, exactly one of the following is true:
\begin{itemize}
\item $a$ is incident to exactly edge in $\RED$ and to exactly one edge in $\BLUE$;
\item $a$ is not incident to any edge in $\RED\cup\BLUE$.
\end{itemize}
\end{lemma}
\begin{proof}
By Remark~\ref{rmk:set_partitionK} together with Lemma~\ref{lemma:N_bar_j}, for each $a\in P_G(\bar j;\svt)=P_G(\bar j,\svt')$, there is  exactly one $b\in P_G^-$ (resp., $b'\in P_G^-$) such that $(b\to a)\in\RED$ (resp., $(b'\to a)\in \BLUE$). For $a\notin P_G(\bar j;\svt)$, there is no such $b$ (resp., $b'$).
\end{proof}

For each $b\in P_G^-\cap\Z^m$, define $N(b):=N_{\bar u}(T_\svt(b))$ and $N'(b):=N_{\bar u}(T_{\svt'}(b))$. Thus for $b\in P_G^-\cap\Z^m$, we have
\[N(b)=\{i\in [m]\mid (b\toidot)\in \RED\}, \quad N'(b)=\{i\in [m]\mid (b\toidot)\in \BLUE\}.\]
We claim that if $|N(b)|=|N'(b)|$ for all $b\in P_G^-\cap\Z^m$ then $\RED=\BLUE$, or equivalently, $N(b)=N'(b)$ for all $b\in P_G^-\cap\Z^m$. 

Let us denote $\RB:=\RED\setminus \BLUE$ and $\BR:=\BLUE\setminus \RED$.


\begin{lemma}\label{lemma:transfer}
  Suppose that for some $c\in\Z^m$ and distinct indices $i,j,k\in [m]$, we have $c+e_i,c+e_j\in P_G^-$. Assume in addition that
  \[(c+e_i\tokdot)\in\RB,\quad (c+e_i\tojdot)\in\BR,\quad (c+e_j\toidot)\in \RB.\]
  Then $(c+e_j\tokdot)\in\RB$.
\end{lemma}
\begin{proof}
  If $c+e_k\notin P_G^-$ then by Axiom~\axHexagon, we have $\svt(c+e_i\tokdot)=\svt(c+e_j\tokdot)$ and $\svt'(c+e_i\tokdot)=\svt'(c+e_j\tokdot)$ so if $(c+e_i\tokdot)\in \RB$ then the same holds for $(c+e_j\tokdot)$. Suppose now that $c+e_k\in P_G^-$. Consider the point $c+e_i+e_k\in P_G$. We know that $(c+e_i\tokdot)\in\RB$, and thus by Remark~\ref{rmk:set_partitionK}, $(c+e_k\toidot)\notin\RED$. Therefore $\svt(c+e_j\toidot)\neq \svt(c+e_k\toidot)$, so by Axiom~\axHexagon we have $\svt(c+e_i\tokdot)=\svt(c+e_j\tokdot)$, and therefore $(c+e_j\tokdot)\in\RED$. The only thing left to show is that $(c+e_j\tokdot)\notin\BLUE$. Indeed, suppose otherwise that $(c+e_j\tokdot)\in\BLUE$. By Remark~\ref{rmk:set_partitionK}, we get that $(c+e_k\tojdot)\notin\BLUE$ which implies $\svt'(c+e_i\tojdot)\neq \svt'(c+e_k\tojdot)$. On the other hand, we also have $\svt'(c+e_i\tokdot)\neq\svt'(c+e_j\tokdot)$ because we know that $(c+e_i\tokdot)\notin\BLUE$. These two conditions together violate Axiom~\axHexagon, and thus $(c+e_j\tokdot)\notin\BLUE$. We are done with the proof.
\end{proof}

For $b\in P_G^-\cap\Z^m$, denote $\RBN(b):=N(b)\setminus N'(b)$ and $\BRN(b)=N'(b)\setminus N(b)$, thus we have $|\RBN(b)|=|\BRN(b)|$ for all $b$. Let us find the point $b^\parr1\in P_G^-\cap\Z^m$ for which  $|\RBN(b^\parr1)|=|\BRN(b^\parr1)|$ is maximal. Choose some $i\in \RBN(b^\parr1)$, thus $(b^\parr1\toidot)\in\RB$. We are going to construct an infinite sequence of points $b^\parr1,b^\parr2,\dots$ such that for all $s\geq 1$, $\RBN(b^\parr s)=\RBN(b^\parr 1)$, $\BRN(b^\parr s)=\BRN(b^\parr 1)$, and $b^\parr{s+1}=b^\parr s+e_i-e_{j^\parr s}$ for some $j^\parr s\in [m]\setminus \{i\}$. Clearly this leads to a contradiction because the $i$-th coordinate of a point in $P_G^-$ cannot be arbitrarily large.

\def\jone{{j^\parr 1}}
\def\dottojone{\dotto \jone}
\def\tojonedot{\todot \jone}
Let us show how to construct the point $b^\parr2$ from $b^\parr1$. By Lemma~\ref{lemma:exactly_one}, there exists a unique index $j^\parr 1\in [m]$ such that $(\dottojone b^\parr1+e_i)\in \BR$, so denote $b^\parr2:=b^\parr1+e_i-e_\jone$. By Lemma~\ref{lemma:exactly_one} again, we have $b^\parr2\in P_G^-\cap\Z^m$.

We claim that $\RBN(b^\parr1)=\RBN(b^\parr2)$ and $\BRN(b^\parr1)=\BRN(b^\parr2)$.  Indeed, let $k\neq i,\jone$ be any element of $\BRN(b^\parr1)$. Letting $c:=b^\parr1-e_\jone$, we have
\[(c+e_\jone\tokdot)\in \BR,\quad (c+e_\jone\toidot)\in\RB,\quad (c+e_i\tojonedot)\in\BR.\]
Then by Lemma~\ref{lemma:transfer} (with $\RB$ and $\BR$ swapped), we must have $(c+e_i\tokdot)\in\BR$, and thus $k\in \BRN(b^\parr2)$. In addition, we know that $i\notin\BRN(b^\parr1)$ and $\jone\in \BRN(b^\parr2)$. It follows that $\BRN(b^\parr1)\subset\BRN(b^\parr2)$, but since $\BRN(b^\parr1)$ has maximal size, we must have $\BRN(b^\parr1)=\BRN(b^\parr2)$. Switching the roles of $b^\parr1$ and $b^\parr2$ and of $i$ and $\jone$, we see that $\RBN(b^\parr1)=\RBN(b^\parr2)$. In particular, we have $i\in \RBN(b^\parr2)$. Repeating this argument for $b^\parr2$, we find $b^\parr3$, etc. As we have noted earlier, constructing such an infinite sequence leads to a contradiction with the assumption that $\RBN(b)$ is non-empty. This finishes the proof of Theorems~\ref{thm:bijections_different} and~\ref{thm:bijections_different_K}.\hfill\qed

\newcommand{\arxiv}[1]{\href{https://arxiv.org/abs/#1}{\textup{\texttt{arXiv:#1}}}}

\bibliographystyle{alpha}
\bibliography{trianguloids}

\begin{thebibliography}{PRRV67}

\bibitem[AB07]{AB}
Federico Ardila and Sara Billey.
\newblock Flag arrangements and triangulations of products of simplices.
\newblock {\em Adv. Math.}, 214(2):495--524, 2007.

\bibitem[AD09]{AD}
Federico Ardila and Mike Develin.
\newblock Tropical hyperplane arrangements and oriented matroids.
\newblock {\em Math. Z.}, 262(4):795--816, 2009.

\bibitem[BB98]{BB}
E.~K. Babson and L.~J. Billera.
\newblock The geometry of products of minors.
\newblock {\em Discrete Comput. Geom.}, 20(2):231--249, 1998.

\bibitem[BCS88]{BCS}
L.~J. Billera, R.~Cushman, and J.~A. Sanders.
\newblock The {S}tanley decomposition of the harmonic oscillator.
\newblock {\em Nederl. Akad. Wetensch. Indag. Math.}, 50(4):375--393, 1988.

\bibitem[BZ93]{BZ}
David Bernstein and Andrei Zelevinsky.
\newblock Combinatorics of maximal minors.
\newblock {\em J. Algebraic Combin.}, 2(2):111--121, 1993.

\bibitem[DS04]{DS}
Mike Develin and Bernd Sturmfels.
\newblock Tropical convexity.
\newblock {\em Doc. Math.}, 9:1--27, 2004.

\bibitem[ES52]{ES}
Samuel Eilenberg and Norman Steenrod.
\newblock {\em Foundations of algebraic topology}.
\newblock Princeton University Press, Princeton, New Jersey, 1952.

\bibitem[FF16]{FF}
Anatoly Fomenko and Dmitry Fuchs.
\newblock {\em Homotopical topology}, volume 273 of {\em Graduate Texts in
  Mathematics}.
\newblock Springer, second edition, 2016.

\bibitem[GKZ08]{GKZ}
I.~M. Gelfand, M.~M. Kapranov, and A.~V. Zelevinsky.
\newblock {\em Discriminants, resultants and multidimensional determinants}.
\newblock Modern Birkh\"auser Classics. Birkh\"auser Boston, Inc., Boston, MA,
  2008.
\newblock Reprint of the 1994 edition.

\bibitem[Hor16]{Horn}
Silke Horn.
\newblock A topological representation theorem for tropical oriented matroids.
\newblock {\em J. Combin. Theory Ser. A}, 142:77--112, 2016.

\bibitem[HRS00]{HRS}
Birkett Huber, J\"org Rambau, and Francisco Santos.
\newblock The {C}ayley trick, lifting subdivisions and the {B}ohne-{D}ress
  theorem on zonotopal tilings.
\newblock {\em J. Eur. Math. Soc. (JEMS)}, 2(2):179--198, 2000.

\bibitem[KT99]{KT}
Allen Knutson and Terence Tao.
\newblock The honeycomb model of {${\rm GL}_n({\bf C})$} tensor products. {I}.
  {P}roof of the saturation conjecture.
\newblock {\em J. Amer. Math. Soc.}, 12(4):1055--1090, 1999.

\bibitem[Kum88]{Kumar}
Shrawan Kumar.
\newblock Proof of the {P}arthasarathy-{R}anga {R}ao-{V}aradarajan conjecture.
\newblock {\em Invent. Math.}, 93(1):117--130, 1988.

\bibitem[Mat89]{Mathieu}
Olivier Mathieu.
\newblock Construction d'un groupe de {K}ac-{M}oody et applications.
\newblock {\em Compositio Math.}, 69(1):37--60, 1989.

\bibitem[OY11]{OYtom}
Suho Oh and Hwanchul Yoo.
\newblock Triangulations of {$\Delta_{n-1}\times\Delta_{d-1}$} and tropical
  oriented matroids.
\newblock In {\em 23rd {I}nternational {C}onference on {F}ormal {P}ower
  {S}eries and {A}lgebraic {C}ombinatorics ({FPSAC} 2011)}, Discrete Math.
  Theor. Comput. Sci. Proc., AO, pages 717--728. Assoc. Discrete Math. Theor.
  Comput. Sci., Nancy, 2011.

\bibitem[OY15]{OYmatching}
Suho {Oh} and Hwanchul {Yoo}.
\newblock {Matching ensembles (extended abstract).}
\newblock In {\em 27th {I}nternational {C}onference on {F}ormal {P}ower
  {S}eries and {A}lgebraic {C}ombinatorics ({FPSAC} 2015)}, Discrete Math.
  Theor. Comput. Sci. Proc., AO, pages 667--672. Assoc. Discrete Math. Theor.
  Comput. Sci., Nancy, 2015.

\bibitem[Pol89]{Polo}
Patrick Polo.
\newblock Vari\'et\'es de {S}chubert et excellentes filtrations.
\newblock {\em Ast\'erisque}, (173-174):10--11, 281--311, 1989.
\newblock Orbites unipotentes et repr\'esentations, III.

\bibitem[Pos09]{Postnikov}
Alexander Postnikov.
\newblock Permutohedra, associahedra, and beyond.
\newblock {\em Int. Math. Res. Not. IMRN}, (6):1026--1106, 2009.

\bibitem[PRRV67]{PRV}
K.~R. Parthasarathy, R.~Ranga~Rao, and V.~S. Varadarajan.
\newblock Representations of complex semi-simple {L}ie groups and {L}ie
  algebras.
\newblock {\em Ann. of Math. (2)}, 85:383--429, 1967.

\bibitem[San00]{SantosDisc}
Francisco Santos.
\newblock A point set whose space of triangulations is disconnected.
\newblock {\em J. Amer. Math. Soc.}, 13(3):611--637, 2000.

\bibitem[San05]{SantosCayley}
Francisco Santos.
\newblock The {C}ayley trick and triangulations of products of simplices.
\newblock In {\em Integer points in polyhedra---geometry, number theory,
  algebra, optimization}, volume 374 of {\em Contemp. Math.}, pages 151--177.
  Amer. Math. Soc., Providence, RI, 2005.

\bibitem[SZ93]{SZ}
Bernd Sturmfels and Andrei Zelevinsky.
\newblock Maximal minors and their leading terms.
\newblock {\em Adv. Math.}, 98(1):65--112, 1993.

\end{thebibliography}

\end{document}